\documentclass[12pt,twoside]{amsart}
\usepackage{amsfonts,amsmath}
\usepackage{pinlabel}
\usepackage{enumerate}
\usepackage{mathtools}
\usepackage{xcolor,colortbl}
\usepackage{url}
\usepackage{cleveref}

\def\endpf{\relax\ifmmode\expandafter\endproofmath\else
  \unskip\nobreak\hfil\penalty50\hskip.75em\hbox{}\nobreak\hfil\bull
  {\parfillskip=0pt \finalhyphendemerits=0 \bigbreak}\fi}
\def\bull{\vbox{\hrule\hbox{\vrule\kern3pt\vbox{\kern6pt}\kern3pt\vrule}\hrule}}

\newtheorem{defn}{Definition}[section]
\newtheorem{lemma}[defn]{Lemma}

\newtheorem{remark}[defn]{Remark}
\newtheorem{proposition}[defn]{Proposition}

\newtheorem{maintheorem}{Theorem}

\newcommand{\nn}{{\mathbb N}}

\newcommand{\spin}{\ifmmode{\rm Spin}\else{${\rm spin}$\ }\fi}
\newcommand{\spinc}{\ifmmode{{\rm Spin}^c}\else{${\rm spin}^c$\ }\fi}

\newcommand{\W}{\mathcal{W}}
\newcommand{\mcS}{\mathcal{S}}

\newcommand{\CP}{\mathbb{CP}}

\definecolor{Gray}{gray}{0.8}

\newenvironment{narrow}[2]{%
 \begin{list}{}{%
  \setlength{\topsep}{0pt}%
  \setlength{\leftmargin}{#1}%
  \setlength{\rightmargin}{#2}%
  \setlength{\listparindent}{\parindent}%
  \setlength{\itemindent}{\parindent}%
  \setlength{\parsep}{\parskip}%
 }%
\item[]}{\end{list}}

\newif\ifpic
\picfalse    

\begin{document}

\title{Equivariant embeddings of rational homology balls}
\author[Brendan Owens]{Brendan Owens}
\address{School of Mathematics and Statistics \newline\indent 
University of Glasgow \newline\indent 
Glasgow, G12 8SQ, United Kingdom}
\email{brendan.owens@glasgow.ac.uk}
\date{\today}
\thanks{}

\begin{abstract}  We generalise theorems of Khodorovskiy and Park-Park-Shin, and give new topological proofs of those theorems, using embedded surfaces in the 4-ball and branched double covers.  These theorems exhibit smooth codimension-zero embeddings of certain rational homology balls bounded by lens spaces.
\end{abstract}

\maketitle

\pagestyle{myheadings}
\markboth{BRENDAN OWENS}{EQUIVARIANT EMBEDDINGS OF RATIONAL HOMOLOGY BALLS}


\section{Introduction}
\label{sec:intro}

The rational blow-down operation was introduced by Fintushel and Stern in \cite{fs} and has been a useful tool in constructing small exotic 4-manifolds; see for example \cite{park,pss,ss}.  The basic setup is that one has two 4-manifolds $C$ and $B$ with diffeomorphic boundary $Y$, and with $B$ a rational homology 4-ball.  Given a closed smooth 4-manifold $X$ containing $C$, the manifold $Z=X\setminus C\,\cup_Y B$ is a rational blow-down of $X$.  For certain favourable examples of $C$ and $B$, this operation  preserves many properties of the smooth structure of $X$, including in particular nonvanishing of Seiberg-Witten invariants.  The most important examples of triples $(C,B,Y)$ for this purpose are as follows.  Let $p>q$ be coprime natural numbers, and let $Y_{p,q}$ be the lens space $L(p^2,pq-1)$.  Let $C_{p,q}$ be the negative-definite plumbed 4-manifold bounded by $Y_{p,q}$.  The lens space $Y_{p,q}$ is known to bound a rational ball $B_{p,q}$.  

One is then interested to know if a 4-manifold $Z$ contains such a submanifold $B$ so that it may be the result of a rational blow-down.  
Khodorovskiy \cite{kho} used Kirby calculus to show that $B_{p,1}$ embeds smoothly in a regular neighbourhood $V_{-p-1}$ of any embedded sphere with self-intersection $-(p+1)$, for $p>1$.  She also showed that for odd $p$, $B_{p,1}$ embeds smoothly in a regular neighbourhood of an embedded sphere with self-intersection $-4$, and hence in $\overline{\CP^2}$.

\begin{maintheorem}[Khodorovskiy \cite{kho}]
\label{thm:Kh12}
For each $p>1$, the rational ball $B_{p,1}$ embeds smoothly in $V_{-p-1}$.
\end{maintheorem}

\begin{maintheorem}[Khodorovskiy \cite{kho}]
\label{thm:Kh13}
For each odd $p>1$, the rational ball $B_{p,1}$ embeds smoothly in $V_{-4}$.
For each even $p>1$, $B_{p,1}$ embeds smoothly in $B_{2,1}\#\overline{\CP^2}.$
\end{maintheorem}

Park-Park-Shin used methods from the minimal model program for complex algebraic 3-folds to generalise  \Cref{thm:Kh12}.  For each $p,q$ they described a linear graph which is roughly speaking half of the negative-definite plumbing graph associated to $C_{p,q}$, and denoted by $Z_{p,q}$  the plumbing of disk bundles over spheres according to this graph.  This $Z_{p,q}$ is called the $\delta$-half linear chain associated to the pair $(p,q)$.

\begin{maintheorem}[Park-Park-Shin \cite{pps}]
\label{thm:pps}
For any $p>q\ge1$, the rational ball $B_{p,q}$ embeds smoothly in the $\delta$-half linear chain $Z_{p,q}$.
\end{maintheorem}

This recovers \Cref{thm:Kh12} since $Z_{p,1}$ is diffeomorphic to the regular neighbourhood of a sphere of self-intersection $-(p+1)$.  Note that if $p>2$ then \Cref{thm:pps} gives two different embeddings, since $B_{p,q}\cong B_{p,p-q}$ but $Z_{p,q}\not\cong Z_{p,p-q}$.

The purpose of this paper is to give relatively simple topological proofs for the Khodorovskiy and Park-Park-Shin results, which also lead to some new embeddings.  Our method is to view $Y_{p,q}$ as the double cover of $S^3$ branched along the two-bridge knot $K_{p,q}=S(p^2,pq-1)$.  The plumbing $C_{p,q}$ is the double cover of $B^4$ branched along the black surface associated to an alternating diagram of $K_{p,q}$, and the rational ball $B_{p,q}$ is the double cover of $B^4$ with branch locus $\Delta_{p,q}$ which is a slice disk (if $p$ is odd) or a disk and a M\"{o}bius band (if $p$ is even).  Using an induction argument inspired by those of Lisca \cite{lisca1,lisca2} we show that $Z_{p,q}$ is the double  cover of the 4-ball branched along a surface $F^\delta_{p,q}$ and that there is a smooth embedding of pairs
$$(B^4_{0.5},\Delta_{p,q})\hookrightarrow(B^4_1,F^\delta_{p,q})$$
where $B^4_r$ is the ball of radius $r$.  We say that $\Delta_{p,q}$ is a {\em sublevel surface} of $F^\delta_{p,q}\subset B^4$.  Taking branched double covers then gives the required embedding of the rational ball $B_{p,q}$.  The proof also yields a clearer understanding of why the $\delta$-half linear plumbing shows up in  \Cref{thm:pps}:  each $\Delta_{p,q}$ is obtained by placing a band exactly half-way along the two-bridge diagram.  \Cref{fig:75embed} illustrates this for $(p,q)=(7,5)$.

\begin{figure}[htbp]
\centering
\includegraphics[width=\textwidth]{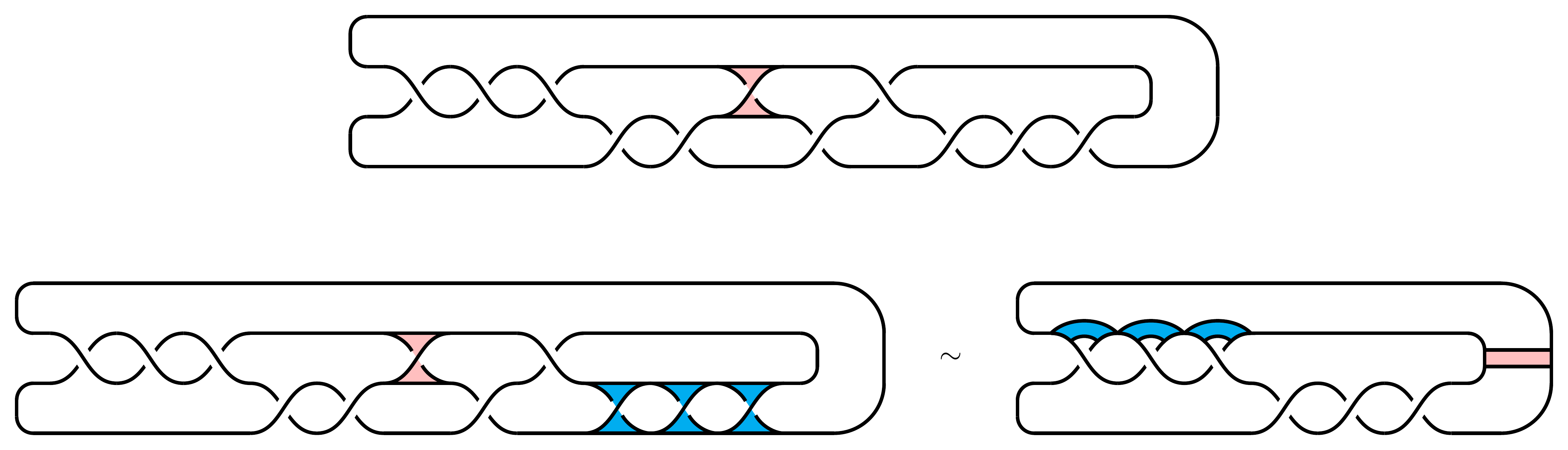}
\begin{narrow}{0.3in}{0.3in}
\caption
{{\bf The $\delta$-half embedding.} The top diagram represents the slice disk $\Delta_{7,5}$.  There are 5 crossings on either side of the pink band.  The bottom diagrams exhibit the slice disk as a sublevel surface of the $\delta$-half surface $F^\delta_{7,5}$.  The equivalence of the bottom two diagrams is explained in  \Cref{sec:proofs}.}
\label{fig:75embed}
\end{narrow}
\end{figure}

We give a proof along similar lines for  \Cref{thm:Kh13}.  The method may also be used to find further embeddings of these rational balls.  We give the following examples.

\begin{maintheorem}
\label{thm:psquared}
For each $p>1$, the slice surface $\Delta_{p^2,p-1}$ is a sublevel surface of the boundary sum $\Delta_{p,1}\natural P_{-}$,
where $P_{-}$ is an unknotted M\"{o}bius band.  Taking double branched covers yields a smooth embedding
$$B_{p^2,p-1}\hookrightarrow B_{p,1}\#\overline{\CP^2}.$$
\end{maintheorem}

\begin{maintheorem}
\label{thm:CP2bar}
For each $n\ge0$, the slice surface $\Delta_{F(2n+2),F(2n)}$ is a sublevel surface of the unknotted M\"{o}bius band $P_{-}$,
where $F(n)$ is the $n$th Fibonacci number.  Taking double branched covers yields a smooth embedding
$$B_{F(2n+2),F(2n)}\hookrightarrow \overline{\CP^2}.$$
\end{maintheorem}

It is interesting to compare \Cref{thm:CP2bar} with the results of \cite{es, hp}, from which it follows that $B_{F(2n+1),F(2n-1)}$ embeds smoothly in $\CP^2$ for each natural number $n$.

In the final section of the paper, we follow \cite{kho,pps} and show that all of the embeddings listed above are {\em simple}.  Essentially this means they are not useful for constructing interesting 4-manifolds.  Our justification for this paper, then, is that it provides useful worked examples of a simple and natural way to realise 4-manifold embeddings.  It also seems interesting  that the result of \cite{pps}, found using birational morphisms of 3-dimensional complex algebraic varieties, turns out to have an explanation on the level of two-bridge knot diagrams.

The paper contains further results of independent interest, including a careful proof that various descriptions of $B_{p,q}$ -- as a symplectic filling of the tight contact structure on $Y_{p,q}$ coming from the universal cover, as the Milnor fibre of a cyclic quotient singularity, or as the double cover of a slice disk described by Casson and Harer -- are the same up to diffeomorphism.  We also provide a new recursive description of the plumbing graphs associated to Wahl singularities, which is related to the Stern-Brocot tree structure on the rational numbers.

\vskip2mm
\noindent{\bf Acknowledgements.}  
The author thanks Paolo Aceto, Marco Golla, Yank\i\ Lekili, Paolo Lisca and Jongil Park for helpful conversations.  We also thank the referee for helpful suggestions to improve the exposition.



\section{Rational numbers and rooted binary trees}
\label{sec:trees}

In this section we describe two trees which arise in the study of Wahl singularities and rational balls bounded by lens spaces; one of them is in fact the well-known Stern-Brocot tree, used for finding efficient rational approximations for real numbers \cite{aigner,brocot,gkp,stern}.  The motivating fact is the following: cyclic quotient singularities of type $\frac{1}{p^2}(1,pq-1)$ correspond to sequences of integers $[c_1,\dots,c_k]$.  The set of all such sequences is obtained from the one-term sequence $[4]$ by the recursive rule described in \cite{wahl}:
$$[c_1,\dots,c_k]\mapsto[c_1+1,\dots,c_k,2]\mbox{ or }[2,c_1,\dots,c_k+1].$$
We will see  that there is an alternative recursive rule which may be used: 
we obtain all sequences from $[2+2]$ via
\begin{align*}
[a_1,\dots,a_{k-1},a_k+b_l,b_{l-1},\dots,b_1]&\mapsto[a_1,\dots,a_{k-1},(a_k+1)+2,b_l,b_{l-1},\dots,b_1]\\
&\mbox{ or }[a_1,\dots,a_{k-1},a_k,2+(b_l+1),b_{l-1},\dots,b_1].
\end{align*}
These two recursions give rise to different labellings of a tree by pairs $(p,q)$, which we now describe.

An (infinite, complete, rooted) binary tree is a tree with a single root and such that each node has two children.  
We will consider labellings of the nodes by pairs $(p,q)$ of coprime natural numbers with $p>q$.

For example we take the tree $\W_1$, which we call the {\em inverse Stern-Brocot tree}.  The root is labelled $(2,1)$ and subsequent nodes are labelled according to the following recursive rule:

{\centering
\includegraphics[width=7cm]{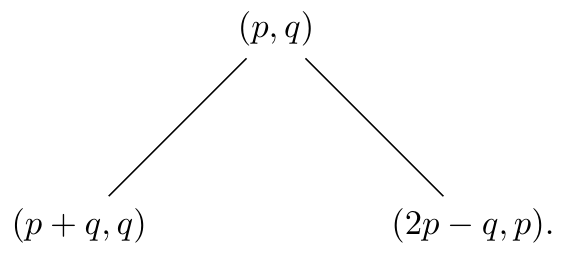}
\par}
\noindent
The first three rows of the tree, and two further nodes, are shown below.

{\centering
\includegraphics[width=8cm]{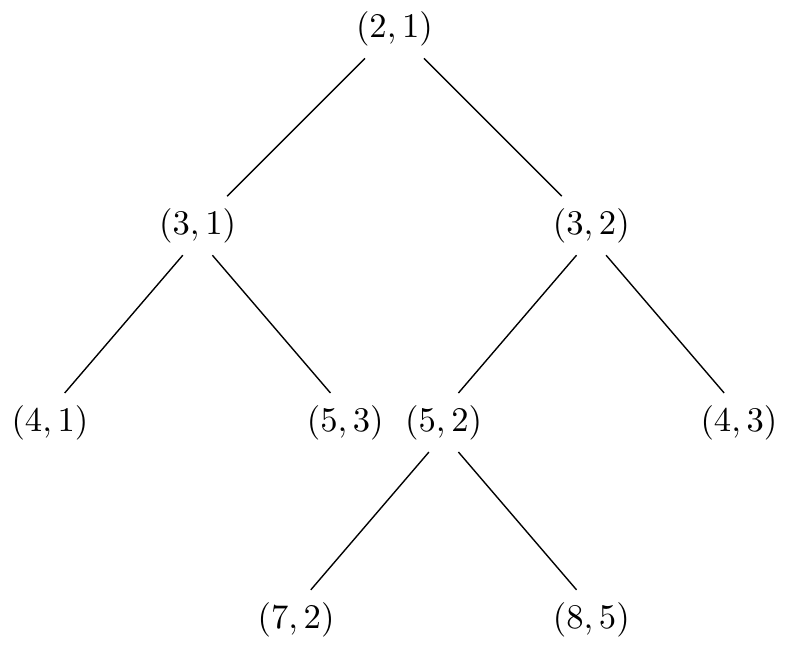}
\par}
\noindent
It is easy to see that there is one node for each coprime pair $(p,q)$ with $p>q$.

The {\em Stern-Brocot tree}, here denoted $\W_2$, has a similar description but at each point the pair $(p,q)$ from $\W_1$ is replaced with $(p,q^{-1}\mod p)$.  Thus the root is again labelled $(2,1)$ and the recursive rule is

{\centering
\includegraphics[width=7cm]{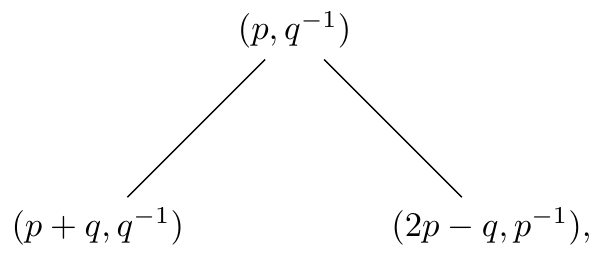}
\par}

\noindent where in each case $(m,n^{-1})$ is short hand for  $(m,n^{-1}\mod m)$.
The first few rows are as follows.

{\centering
\includegraphics[width=8cm]{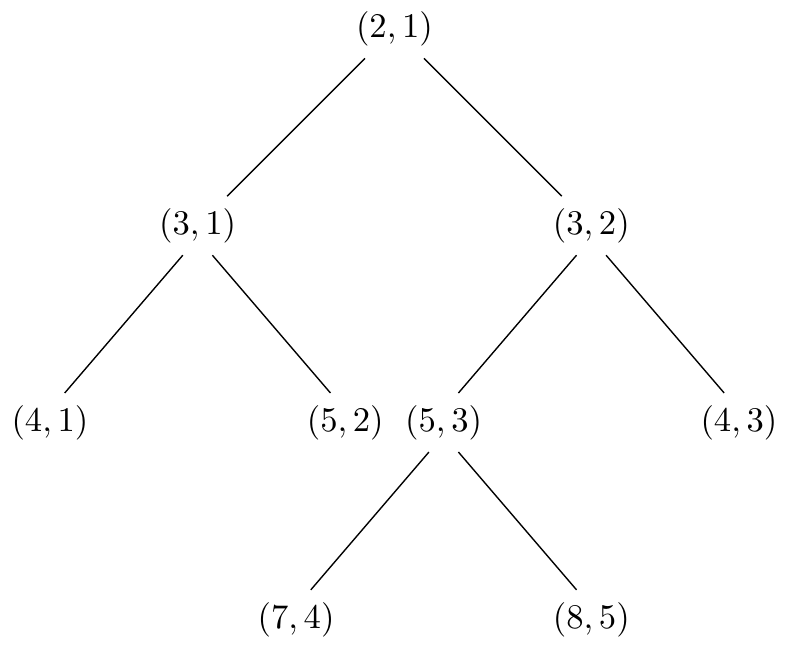}
\par}

\begin{remark}
It is not a difficult exercise to show that this is the same as the recursion rule described in \cite{aigner,gkp}.  For example one may write $lp+mq=1,$ with $m=q^{-1}\mod p$ and then use $l$ and $m$ to find expressions for various other inverses involved in the recursion rule.
\end{remark}

\begin{remark} In fact the tree $\W_2$ is a subtree of the usual Stern-Brocot tree.  The full tree has root labelled $(1,1)$, and is easily recovered from the tree described here.
\end{remark}


\section{Continued fractions and rational balls}
\label{sec:wahl}
In this section we  set our conventions for two-bridge links and lens spaces.  We also describe a family of slice surfaces $\Delta_{p,q}$; we further verify that these are the same as those described by Casson and Harer in \cite{ch}, and that the double cover of $B^4$ branched along $\Delta_{p,q}$ is the same rational ball $B_{p,q}$ referred to in each of \cite{fs,kho,lm,pps,park,sym,wahl}.

We use Hirzebruch-Jung continued fractions, with the following notation:
$$[a_1,a_2,\ldots,a_k]
:=a_1-\frac{1}{a_2-\raisebox{-3mm}{$\ddots$
\raisebox{-2mm}{${-\frac{1}{\displaystyle{a_k}}}$}}}.$$

Given a pair of coprime natural numbers $p$ and $q$ with $p>q$, let
$$p/q=[a_1,\dots,a_k],$$
with each $a_i\ge2$.
We now describe a planar graph $\Gamma_{p,q}$ with $k+1$ vertices $v_0,\dots,v_k$.  For each $1\le i< k$, this graph has a single edge joining $v_i$ to $v_{i+1}$.  For each $i>0$, there are edges between $v_0$ and $v_i$ so that the valence of $v_i$ is $a_i$.  We draw the first set of edges in a line, and the second set of edges all on the same side of this line.  An example is shown in \Cref{fig:Spq}.  We then draw a link diagram for which $\Gamma_{p,q}$ is the Tait graph; this is obtained from the unlink diagram which is the boundary of a tubular neighbourhood of the vertices by adding a right-handed crossing along each edge.  See for example \Cref{fig:Spq}.  We refer to this as the standard alternating diagram of the two-bridge link $S(p,q)$.  Note that according to this convention $S(3,1)$ is the left-handed trefoil.
We define the lens space $L(p,q)$ to be the double cover of $S^3$ branched along $S(p,q)$. 

\begin{figure}[htbp]
\centering
\includegraphics[width=\textwidth]{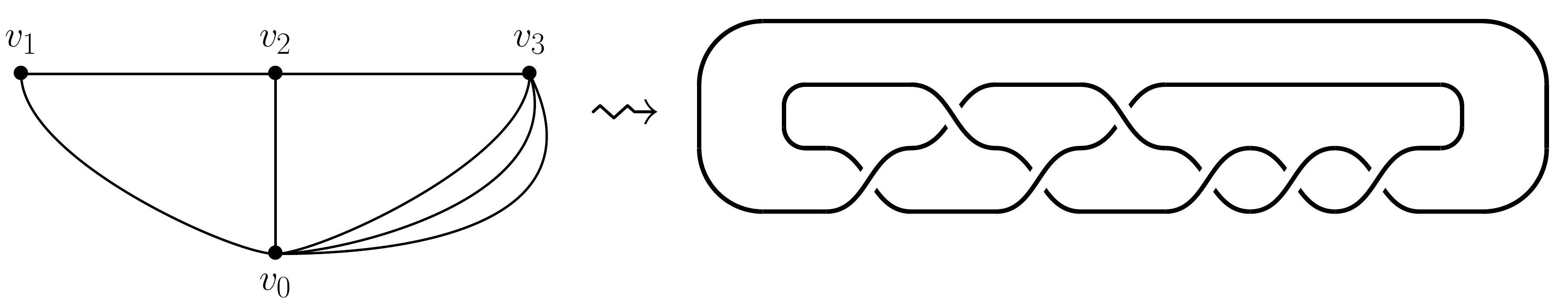}
\begin{narrow}{0.3in}{0.3in}
\caption
{{\bf The planar graph $\Gamma_{18,11}$ and the two-bridge link $S(18,11)$.} Note that $18/11=[2,3,4]$.}
\label{fig:Spq}
\end{narrow}
\end{figure}

It will be convenient to recall how to pass between the continued fractions for $p/q$ and for $p/(p-q)$.  This can be  done using the Riemenschneider point rule, as described in \cite{lisca1,pps}.  We can also use planar graphs: the planar dual of $\Gamma_{p,q}$ is $\Gamma_{p,p-q}$, with the continued fraction coefficients read from left to right for both.  An example is shown in \Cref{fig:planardual}.

\begin{figure}[htbp]
\centering
\includegraphics[width=\textwidth]{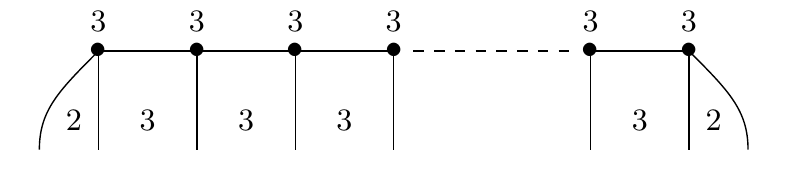}
\begin{narrow}{0.3in}{0.3in}
\caption
{{\bf Continued fractions and planar graphs.}  This example shows that if 
$p/q=[3^n]$, then $p/(p-q)=[2,3^{n-1},2]$.  We have drawn $v_0$ ``at infinity'' for convenience.}
\label{fig:planardual}
\end{narrow}
\end{figure}

The following lemma about continued fractions is presumably known to experts (cf. \cite[Lemma 8.5]{hp}, \cite[Remark 3.2]{lec}).

\begin{lemma}
\label{lem:cf}
Let $p>q$ be coprime natural numbers.  Suppose that
$$\frac{p}{q}=[a_1,\dots,a_k]$$
and
$$\frac{p}{p-q}=[b_1,\dots,b_l].$$
Then
$$\frac{p^2}{pq-1}=[a_1,\dots,a_{k-1},a_k+b_l,b_{l-1},\dots,b_1],$$
and
$$\frac{p^2}{pq+1}=[a_1,\dots,a_{k-1},a_k,2,b_l,b_{l-1},\dots,b_1],$$
\end{lemma}

\begin{proof}
We use structural induction via the inverse Stern-Brocot tree, or in other words based on the fact that each ordered coprime pair of natural numbers may be obtained from $(2,1)$ via a finite sequence of the following steps:
$$(p,q)\mapsto(p+q,q)$$
and
$$(p,q)\mapsto(2p-q,p).$$
The statement holds for the base case and the inductive step follows easily from \cite[Lemma 9.1]{lisca1} together with the following identities:
 \begin{align}\label{eq:cf}
  \begin{aligned}
   \frac{p+q}{q}&=1+\frac{p}{q}=[a_1+1,\dots,a_k], \\       \frac{2p-q}{p}&=2-\frac{1}{p/q}=[2,a_1,\dots,a_k],
  \end{aligned}
  &&
  \begin{aligned}
   \frac{p+q}{p}&=2-\frac{1}{p/(p-q)}=[2,b_1,\dots,b_l], \\       \frac{2p-q}{p-q}&=1+\frac{p}{p-q}=[b_1+1,\dots,b_l].
  \end{aligned}
 \end{align}
\end{proof}

We recall one more well-known fact about continued fractions.  Note this gives the well-known isotopy $S(p,q)=S(p,q^{-1}\mod p)$.

\begin{lemma}
\label{lem:cfinv}
Let $p>q$ be coprime natural numbers.  If
$$\frac{p}{q}=[a_1,\dots,a_k],$$
then
$$\frac{p}{q^{-1}\mod p}=[a_k,\dots,a_1].$$
\end{lemma}
\begin{proof}
This can be proved by induction on $k$.  For details, see for example \cite{hnk}.
\end{proof}

We pause to justify the statement made at the beginning of \Cref{sec:trees}.

\begin{lemma}
\label{lem:wahl}
Let $\mcS_1$ denote the set of strings obtainable from $[4]$ by iteration of 
$$[c_1,\dots,c_k]\mapsto[c_1+1,\dots,c_k,2]\mbox{ or }[2,c_1,\dots,c_k+1].$$
Let $\mcS_2$ denote the set of strings obtainable from $[4]=[2+2]$ via
\begin{align*}
[a_1,\dots,a_{k-1},a_k+b_l,b_{l-1},\dots,b_1]&\mapsto[a_1,\dots,a_{k-1},(a_k+1)+2,b_l,b_{l-1},\dots,b_1]\\
&\mbox{ or }[a_1,\dots,a_{k-1},a_k,2+(b_l+1),b_{l-1},\dots,b_1].
\end{align*}
Then $\mcS_1=\mcS_2$ is the set of Hirzebruch-Jung continued fraction expansions of $$\left\{\frac{p^2}{pq-1}\,:\,p>q>0, (p,q)=1\right\}.$$

Similarly the set of continued fraction expansions of $$\left\{\frac{p^2}{pq+1}\,:\,p>q>0, (p,q)=1\right\}$$
is obtainable from $[2,2,2]$ either via 
$$[c_1,\dots,c_k]\mapsto[c_1+1,\dots,c_k,2]\mbox{ or }[2,c_1,\dots,c_k+1]$$
or via
\begin{align*}
[a_1,\dots,a_{k-1},a_k,2,b_l,b_{l-1},\dots,b_1]&\mapsto[a_1,\dots,a_{k-1},(a_k+1),2,2,b_l,b_{l-1},\dots,b_1]\\
&\mbox{ or }[a_1,\dots,a_{k-1},a_k,2,2,(b_l+1),b_{l-1},\dots,b_1].
\end{align*}
\end{lemma}
\begin{proof}
The proof is by induction, using the fact that each coprime pair $(p,q)$ with $p>q>0$ appears exactly once as a node label in each of the trees $\W_1$ and $\W_2$ described in \Cref{sec:trees}.  For the base case of the first statement, the root of both trees is labelled $(2,1)$ and the continued fraction expansion of $2^2/(2\cdot1-1)$ is $[4]$.  The inductive step follows from Lemmas \ref{lem:cf} \and \ref{lem:cfinv} using \eqref{eq:cf} and is left as an exercise for the reader.  The second statement is similar. 
\end{proof}

We now describe a family of slice surfaces $\Delta_{p,q}$ bounded by the links $K_{p,q}:=S(p^2,pq-1)$.  The first such is $\Delta_{2,1}$, shown in \Cref{fig:Delta21}; applying the band move shown in pink converts the diagram to one of the two-component unlink.

\begin{figure}[htbp]
\centering
\includegraphics[width=\textwidth]{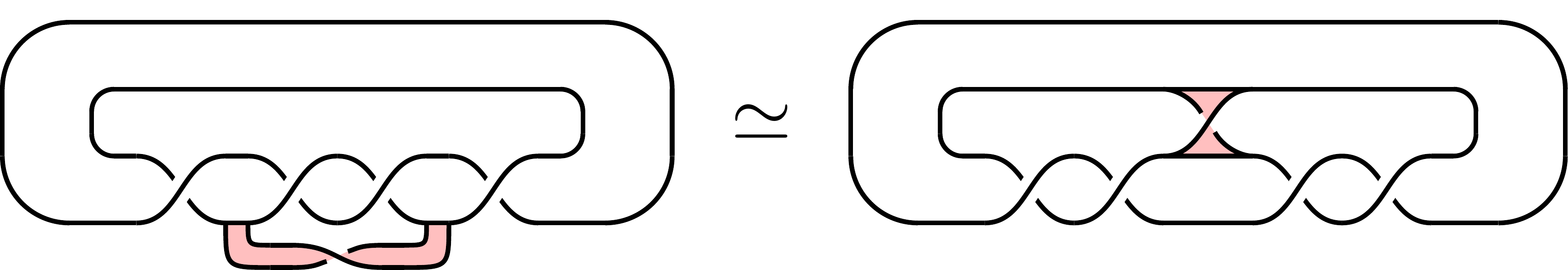}
\begin{narrow}{0.3in}{0.3in}
\caption
{{\bf The ribbon surface $\Delta_{2,1}$.}}
\label{fig:Delta21}
\end{narrow}
\end{figure}

There are  two ways to recursively build the family $\Delta_{p,q}$.  Starting with the left diagram in \Cref{fig:Delta21}, we apply the recursive rule indicated in \Cref{fig:exp2}.  Alternatively we may start with the right diagram in  \Cref{fig:Delta21} and apply the recursion from \Cref{fig:exp1}. 

\begin{figure}[htbp]
\centering
\includegraphics[width=\textwidth]{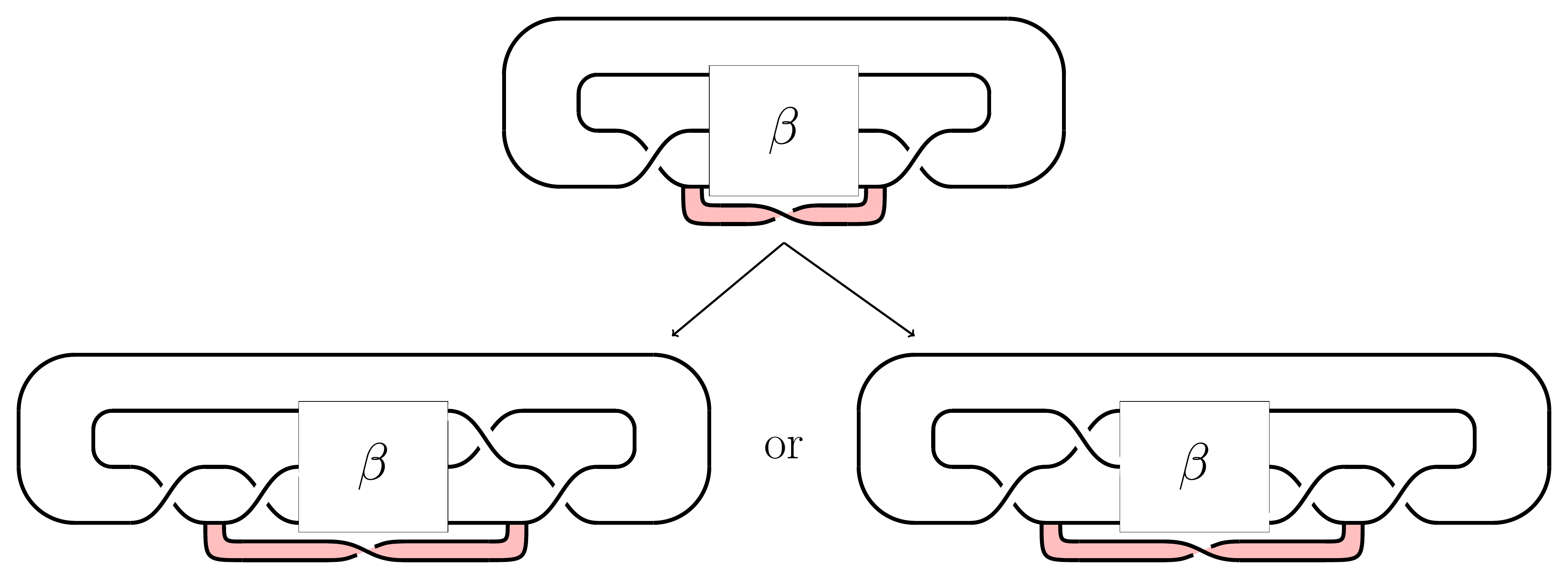}
\begin{narrow}{0.3in}{0.3in}
\caption
{{\bf Moving down the inverse Stern-Brocot tree.}  The box marked $\beta$ contains a 3-braid.}
\label{fig:exp2}
\end{narrow}
\end{figure}

\begin{figure}[htbp]
\centering
\includegraphics[width=11cm]{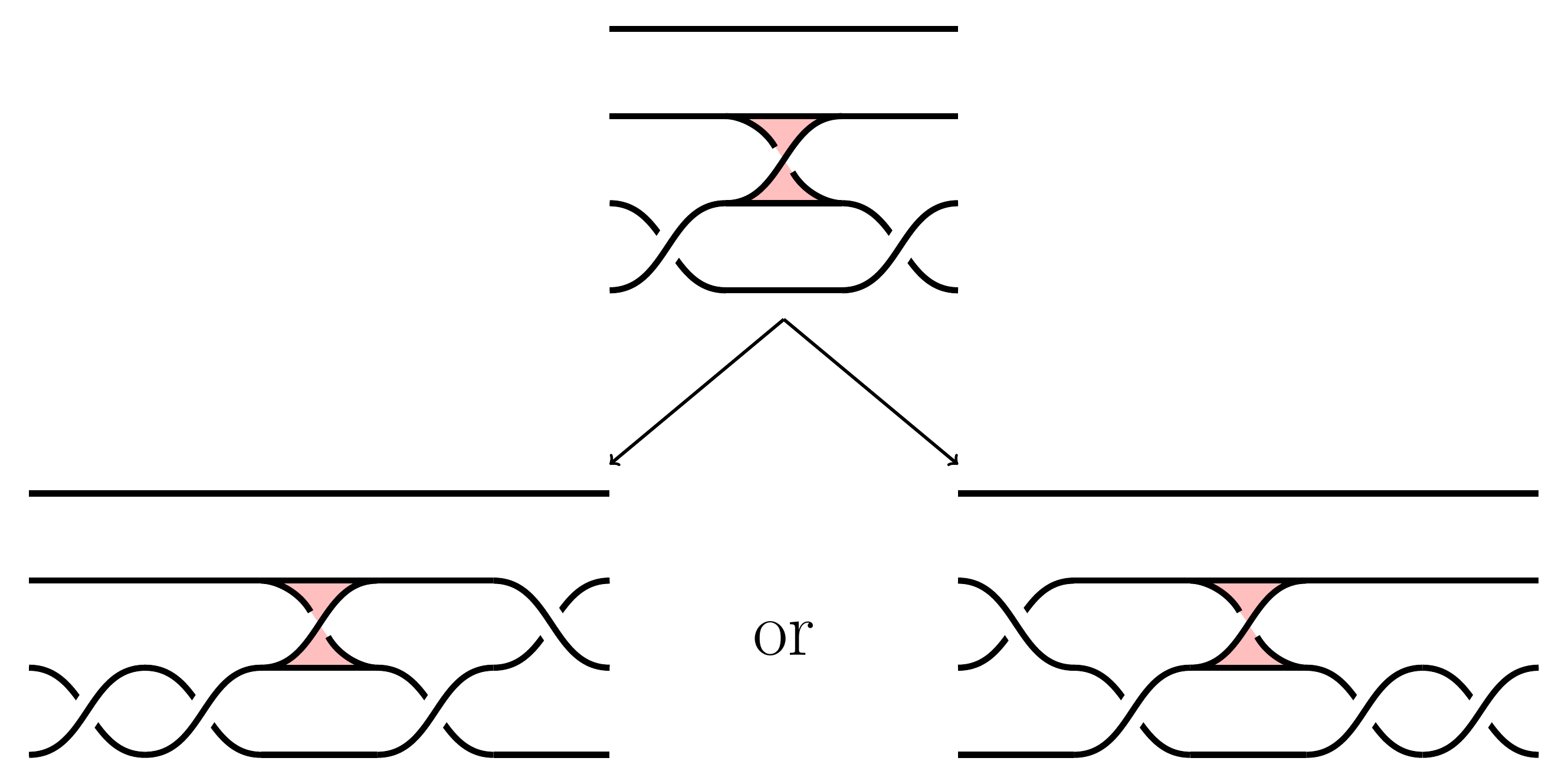}
\begin{narrow}{0.3in}{0.3in}
\caption
{{\bf Moving down the Stern-Brocot tree.}}
\label{fig:exp1}
\end{narrow}
\end{figure}

In either case we see that if the band move indicated in the top diagram converts the link to a two-component unlink, then the same is true for each of the diagrams below.  Suppose that the link in the top diagram in either case is a two-bridge link corresponding to the continued fraction
$$[a_1,\dots,a_{k-1},a_k+b_l,b_{l-1},\dots,b_1].$$
Then the two lower diagrams in \Cref{fig:exp2} are the standard alternating diagrams of the two-bridge links corresponding to
$$[a_1+1,\dots,a_{k-1},a_k+b_l,b_{l-1},\dots,b_1,2]$$
and
$$[2,a_1,\dots,a_{k-1},a_k+b_l,b_{l-1},\dots,b_1+1],$$
respectively, and the two lower diagrams in \Cref{fig:exp1} are the standard diagrams of the two-bridge links corresponding to
$$[a_1,\dots,a_{k-1},(a_k+1)+2,b_l,b_{l-1},\dots,b_1]$$
and
$$[a_1,\dots,a_{k-1},a_k,2+(b_l+1),b_{l-1},\dots,b_1]$$
respectively.

Comparing with \Cref{lem:cf}, we see that if the top diagram of \Cref{fig:exp2} represents a band move converting $K_{p,q}$ to the two-component unlink, then the bottom two diagrams represent such a band move for $K_{p',q'}$, where $(p',q')=(p+q,q)$ on the left and $(2p-q,p)$ on the right; this corresponds to the recursive rule for the inverse Stern-Brocot tree $\W_1$ from \Cref{sec:trees}.  Recursively we obtain a ribbon surface $\Delta_{p,q}$ bounded by each $K_{p,q}$, which has two zero-handles and a single 1-handle.

Similarly if the top diagram of \Cref{fig:exp1} represents a band move converting $K_{p,q^{-1}}$ to the two-component unlink, then the bottom two diagrams represent such a band move for $K_{p',q'^{-1}}$, where $(p',q')=(p+q,q)$ on the left and $(2p-q,p)$ on the right; this corresponds to the recursive rule for the  Stern-Brocot tree $\W_2$ from \Cref{sec:trees}.  We again obtain a ribbon surface $\Delta'_{p,q}$ bounded by each $K_{p,q}$, which has two zero-handles and a single 1-handle.

We observe  by induction (using either \Cref{fig:exp2} or \Cref{fig:exp1} for the inductive step) that the bands giving the two surfaces have their ends on the same component of $K_{p,q}$, and if both band moves are applied one after the other then the second one is the standard band move converting the two-component unlink to the three-component unlink.
This shows that the surface given by the pair of band moves is obtained from either $\Delta_{p,q}$ and $\Delta'_{p,q}$ by adding a cancelling pair of critical points, and thus that the slice surfaces $\Delta_{p,q}$ and $\Delta'_{p,q}$ are isotopic to each other. 

It is straightforward to draw the bands described above for a particular example.  Begin by drawing the standard alternating diagram of $K_{p,q}=S(p^2,pq-1)$ as described above.  Then as in \Cref{fig:Delta75}, the band obtained from \Cref{fig:exp2} goes horizontally across the bottom of the diagram, attached just inside the last crossing at each end.  The band obtained from \Cref{fig:exp1} is placed vertically, half way along the diagram, with the same number of crossings on either side.  Note in particular that the number of crossings in each region to the right and left of the vertical band may be read off from the continued fraction expansions of $p/q$ and $p/(p-q)$ as in \Cref{lem:cf}. To distinguish between these band moves we will refer to them from now on as the horizontal band and the vertical band associated to $\Delta_{p,q}$.

\begin{figure}[htbp]
\centering
\includegraphics[width=14cm]{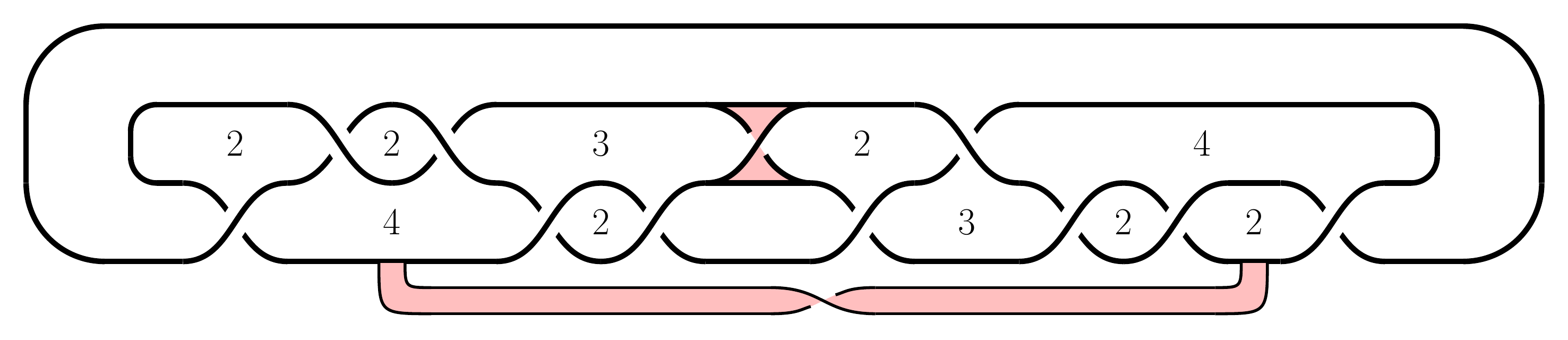}
\begin{narrow}{0.3in}{0.3in}
\caption
{{\bf The slice disk $\Delta_{7,5}$.}  Each label gives a count of crossings in the labelled region.  Note how these correspond to the continued fraction coefficients of $7/5=[2,2,3]$ and $7/2=[4,2]$ (compare Lemmas \ref{lem:cf} and \ref{lem:cfinv}, noting that $pq+1=-(pq-1)^{-1}\mod p^2$).}
\label{fig:Delta75}
\end{narrow}
\end{figure}

The following lemma and proposition tell us that various constructions of rational balls bounded by $L(p^2,pq-1)$ considered in the literature in relation to rational blow-down are all the same up to diffeomorphism.  This is  known to experts, but nonetheless there is some confusion in the literature, so we sketch a proof.  The author is grateful to Yank\i\ Lekili and Paolo Lisca for helpful conversations on this point.

\begin{lemma}
\label{lem:DeltaB}
Let $p>q$ be coprime natural numbers with 
$$p/q=[a_1,\dots,a_k]\quad\mbox{and}\quad p/(p-q)=[b_1,\dots,b_l],$$
where each $a_i$ and $b_j$ is at least 2.
Then the double cover of $B^4$ branched along $\Delta_{p,q}$ is given by the relative Kirby diagram in \Cref{fig:Wpq}.
\end{lemma}

\begin{proof}  The part of the Kirby diagram shown in black, with bracketed framing indices, is a Kirby diagram for the double cover of the 4-ball branched along one of the chessboard surfaces of the unlink diagram which results from the vertical band move.  The branched double cover of the link cobordism given by inverting the vertical band move is the cobordism given by attaching a 2-handle along the $(-1)$-framed red curve. 
\end{proof}

\begin{figure}[htbp]
\centering
\includegraphics[width=\textwidth]{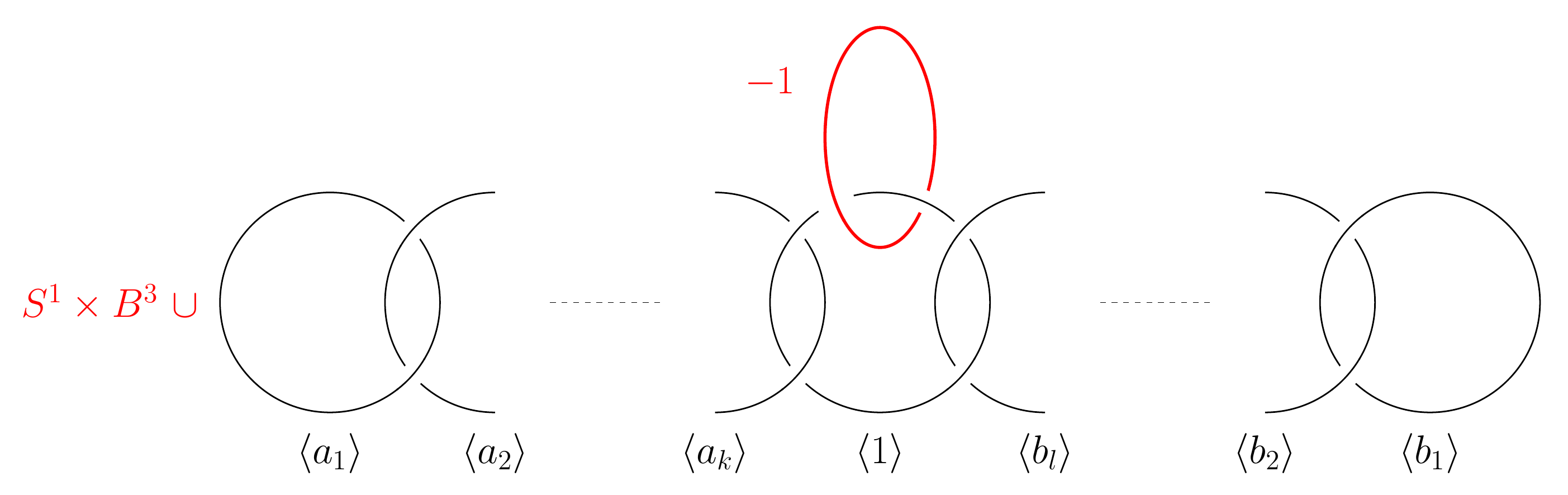}
\begin{narrow}{0.3in}{0.3in}
\caption
{{\bf A Kirby diagram for $W_{p^2,pq-1}(a_1,\dots,a_k,1,b_l,\dots,b_1)$.}  This represents a single 2-handle attached to $S^1\times B^3$.}
\label{fig:Wpq}
\end{narrow}
\end{figure}

\begin{proposition}
\label{prop:DeltaB}
For each coprime pair of natural numbers $p>q$, the slice surface $\Delta_{p,q}$ bounded by $K_{p,q}=S(p^2,pq-1)$ is isotopic to the slice surface described by Casson and Harer in \cite{ch}.  Furthermore the double cover of the 4-ball branched along $\Delta_{p,q}$ is diffeomorphic to the Milnor fibre $B_{p,q}$ of the cyclic quotient singularity of type $\frac{1}{p^2}(1,pq-1)$.
\end{proposition}
\begin{proof}
We first consider the slice surface given by Casson and Harer \cite{ch}.  In the notation of that paper we take $c=-1$, $x=1/0$, $y=-q/p$, and $z=-(p-q)/p$.  Then the third diagram down on \cite[page 32]{ch}, with the crossing in the band shown there changed in order for the band move to yield the unlink, may be seen to be isotopic to that shown in \Cref{fig:DeltaCH}.  The reader may verify, with reference to \Cref{lem:cf}, that this agrees with the description of $\Delta_{p,q}$ given above, with the vertical band.

\begin{figure}[htbp]
\centering
\includegraphics[width=8cm]{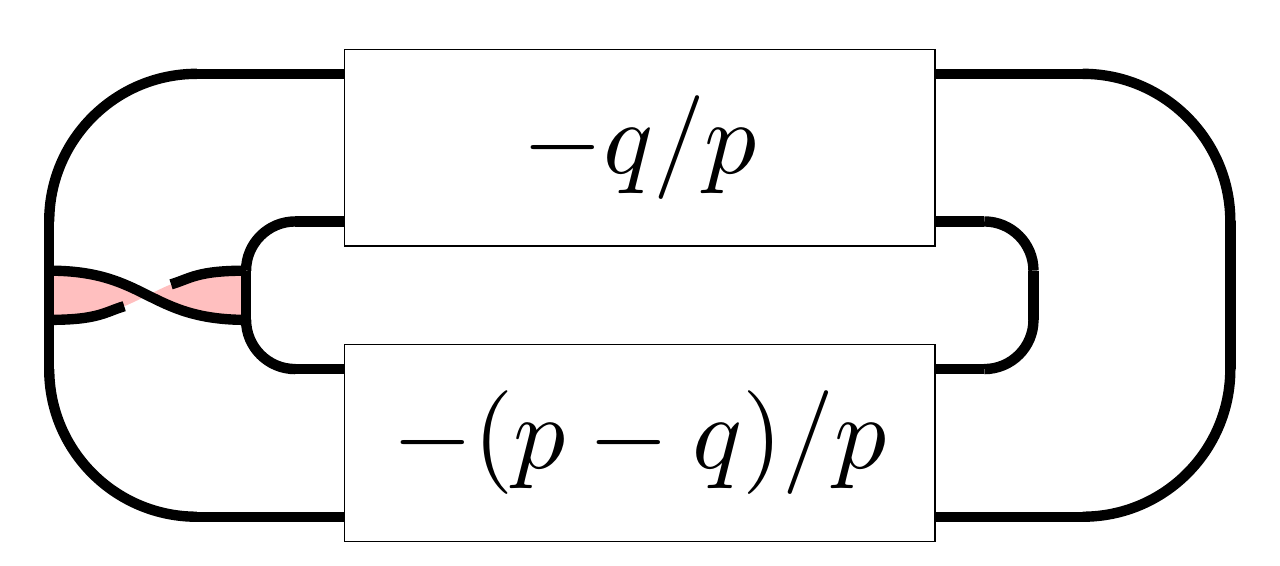}
\begin{narrow}{0.3in}{0.3in}
\caption
{{\bf The slice surface of Casson and Harer.} For the rational tangle notation used in this diagram see \cite[Description 2]{ch}. }
\label{fig:DeltaCH}
\end{narrow}
\end{figure}

The last statement of the Proposition follows, as in \cite[Lemma 3.1]{lm}, from Lisca's classification of symplectic fillings of the tight contact structure $\bar\xi_\mathrm{st}$ on a lens space coming from that on its universal cover $S^3$ \cite{lisca3}.

Using the method of proof of \cite[Corollary 1.2]{lisca3} one may show that it follows from \cite[Theorem 1.1]{lisca3} that there is a unique rational ball symplectic filling of $(L(p^2,pq-1),\bar\xi_\mathrm{st})$ up to diffeomorphism.  This is the manifold
$W_{p^2,pq-1}(a_1,\dots,a_k,1,b_l,\dots,b_1)$
which is shown in \Cref{fig:Wpq}.
We have seen in \Cref{lem:DeltaB} that this is diffeomorphic to the double cover of $B^4$ branched along $\Delta_{p,q}$.  The last statement of the Proposition now follows, since the Milnor fibre of  the cyclic quotient singularity of type $\frac{1}{p^2}(1,pq-1)$ is a rational ball symplectic filling of $(L(p^2,pq-1),\bar\xi_\mathrm{st})$ \cite{lm}.
\end{proof}


\section{Embeddings via double branched covers}
\label{sec:proofs}
In this section we provide proofs for the theorems stated in the introduction.  The proofs will involve manipulations of ``knot with bands" diagrams representing properly embedded surfaces in $B^4$.  These are diagrams consisting of a knot or link $K$ together with a set of bands attached, such that the band moves convert $K$ to an unlink $U$.  As usual, we interpret this as describing a movie for a surface embedding in $B^4$ to which the radial distance function restricts to give a Morse function, with a minimum for each component of $U$ and a saddle for each band.  Maxima of such a surface would result in $K$ being replaced by a union of $K$ and an unlink, but these will not occur in the embeddings we consider.  The knot with bands diagram does not specify the order in which the saddles occur during the movie, but changing this order does not change the isotopy class of the embedding.  By a slight abuse of notation, we will occasionally use the same letter to refer to an embedded surface in $B^4$ or to a knot with bands diagram representing that surface.
Following \cite{frank} we note that the resulting embedded surface is also unchanged up to isotopy by any sequence of {\em band slides} or {\em band swims}.  These moves are shown in \Cref{fig:slideswim}.    Both moves may be interpreted as an isotopy of the blue band coming from the top of the diagram after applying the band move indicated by the pink band; one is then free to slide and swim any band over any other, using the freedom to change the order of the saddle points.
It is shown in \cite{kk,frank} that these moves together with introduction and removal of cancelling pairs gives a complete calculus for 2-knots, but we will not need this here.

\begin{figure}[htbp]
\centering
\includegraphics[width=\textwidth]{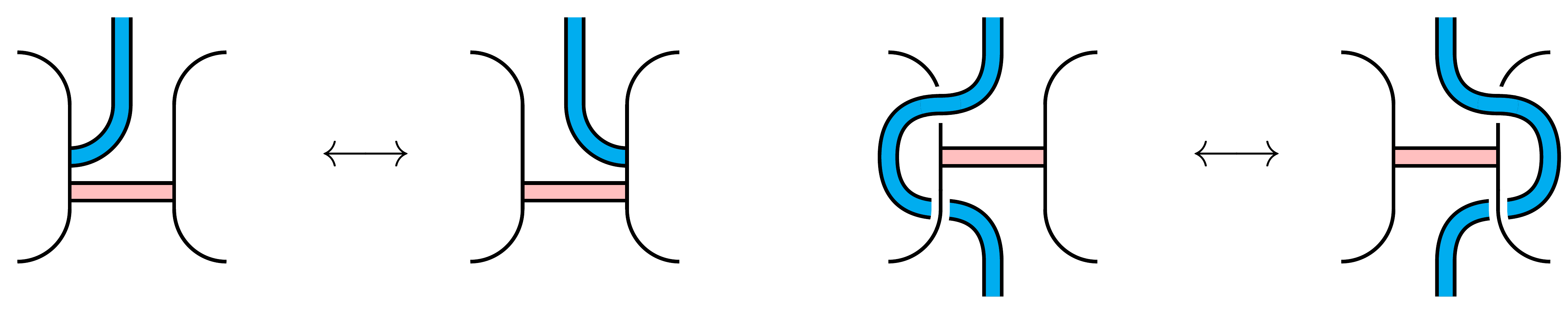}
\begin{narrow}{0.3in}{0.3in}
\caption
{{\bf Band slide and band swim.}}
\label{fig:slideswim}
\end{narrow}
\end{figure}

Given two surfaces $F_1$ and $F_2$ in $B^4$, we say that $F_1$ is a sublevel surface of $F_2$ if there is a smooth embedding of pairs
$$(B^4_{0.5},F_1)\hookrightarrow(B^4_1,F_2)$$
where $B^4_r$ is the ball of radius $r$.  This is equivalent to existence of a movie for $F_2$ whose final scenes consist of a movie for $F_1$.  This in turn may be realised as a knot with bands diagram for $F_2$ which yields a knot with bands for $F_1$ after applying a subset of the band moves (and possibly also removing an unlink corresponding to some minima).

One way to produce properly embedded surfaces in the 4-ball is to take an embedded surface with no closed components in $S^3$ and push its interior inside the 4-ball.  To obtain a knot with bands diagram of this, we choose a handle decomposition of the surface with 0- and 1-handles and attach bands dual to the 1-handles.  To put this another way, we choose a set of properly embedded arcs in the surface that cut it up into a union of disks; neighbourhoods of these arcs give bands.  The examples of the negative M\"{o}bius band $P_-$ and the twisted annulus $F_{-4}$ are shown in \Cref{fig:FP}.
Another example is shown in the bottom two diagrams of \Cref{fig:Fdelta75}; a different knot with bands representation of the same surface is shown in the bottom right diagram of \Cref{fig:75embed}.

\begin{figure}[htbp]
\centering
\includegraphics[width=\textwidth]{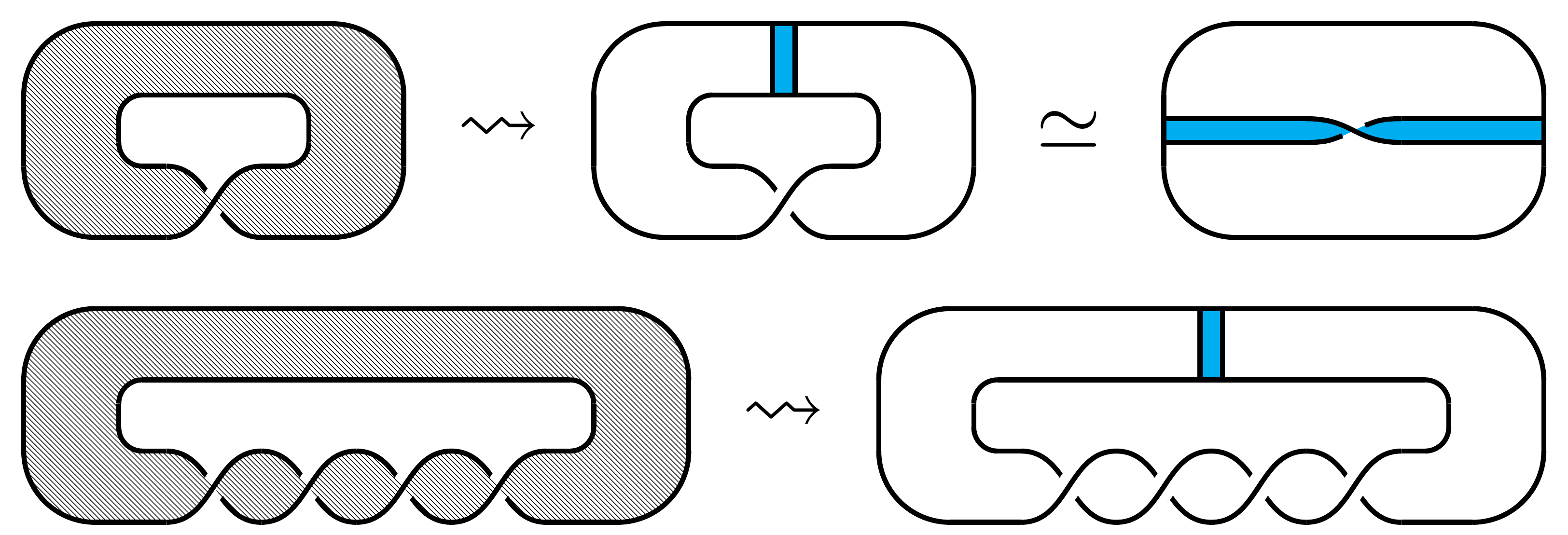}
\begin{narrow}{0.3in}{0.3in}
\caption
{{\bf The surfaces $P_-$ and $F_{-4}$.} The branched double covers are $\overline{\CP^2}$ and $V_{-4}$ respectively.  }
\label{fig:FP}
\end{narrow}
\end{figure}

\begin{proof}[Proof of Theorems \ref{thm:Kh12} and \ref{thm:pps}]
Let $p>q$ be a coprime pair of natural numbers.  We will describe a surface $F^\delta_{p,q}$ whose branched double cover is the $\delta$-half linear chain $Z_{p,q}$, and which has $\Delta_{p,q}$ as a sublevel surface.  \Cref{thm:pps}, and hence also \Cref{thm:Kh12}, then follows on taking branched double covers.

We first describe the $\delta$-half linear chain $Z_{p,q}$.  As usual we take $p/q=[a_1,\dots,a_k]$ with each $a_i\ge2$.  Then $Z_{p,q}$ is the linear plumbing of disk bundles over $S^2$ with weights $[a_1,\dots,a_{k-1},a_k+1].$
This is the double cover of $B^4$ branched along the pushed-in black surface of the standard alternating diagram of the corresponding two-bridge knot, which is $S(p+q^{-1}\mod p,q^{-1}\mod p)$.  This in turn is obtained from the standard diagram for $S(p^2,pq-1)$ by drawing a vertical line through the diagram with two more crossings on the left than on the right, and capping off the portion of the diagram to the left of this line.  We denote this surface by $F^\delta_{p,q}$.  An example is shown in \Cref{fig:Fdelta75}, which also indicates our convention for which is the black chessboard surface.

We next describe a knot with bands representation of a further surface $F'_{p,q}$.  Start with the standard alternating diagram of $K_{p,q}=S(p^2,pq-1)$.  Add the pink vertical band to obtain the slice surface $\Delta_{p,q}$, as in the first diagram in \Cref{fig:75embed}.  Then going to the right from the centre of the diagram, replace all but the first of the crossings corresponding to edges of $\Gamma_{p,q}$  incident to $v_0$ with blue bands, as in the second diagram of \Cref{fig:75embed}.  Call the resulting surface $F'_{p,q}$.  It is clear that this has the same boundary link as $F^\delta_{p,q}$.  We claim that in fact these two embedded surfaces in $B^4$ are isotopic.

\begin{figure}[htbp]
\centering
\includegraphics[width=\textwidth]{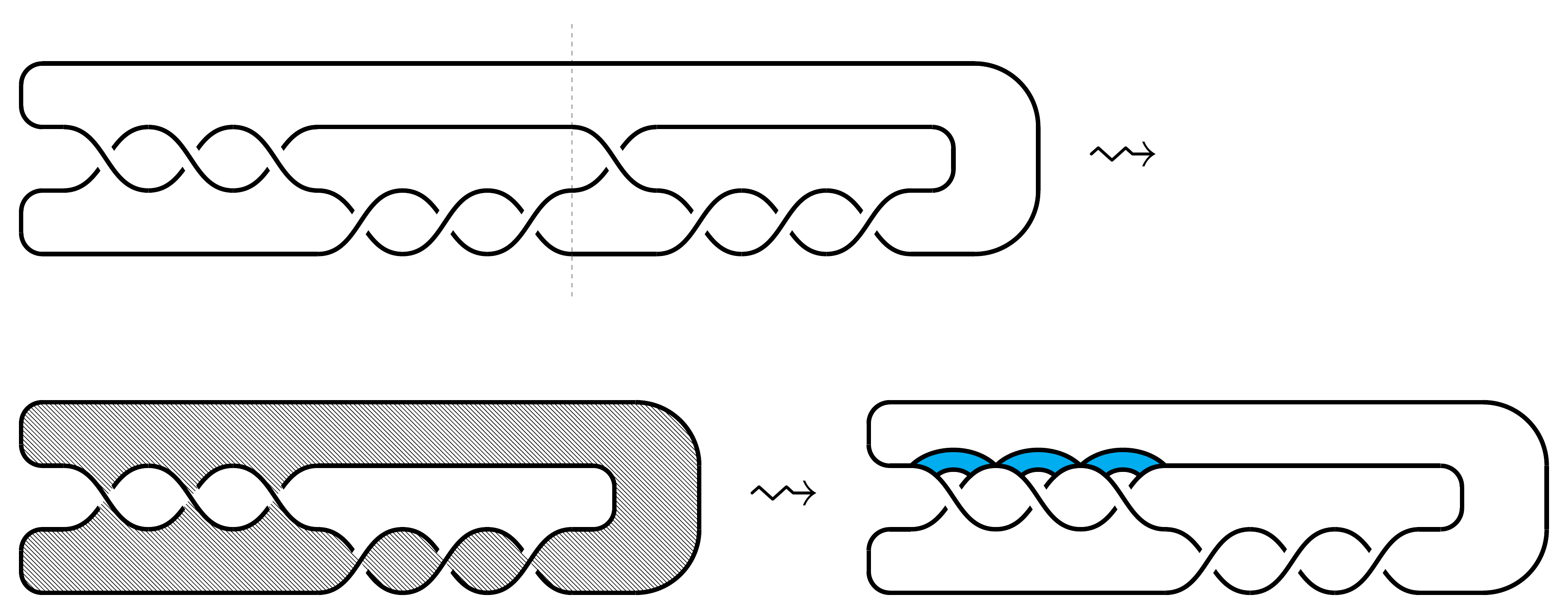}
\begin{narrow}{0.3in}{0.3in}
\caption
{{\bf The surface $F^\delta_{7,5}$.} The top diagram shows $S(49,34)$, which is then cut vertically one crossing to the right of the centre line yielding $S(10,7)$.  
The diagram on the bottom left shows the black surface for $S(10,7)$.  Pushing the interior into the 4-ball yields $F^\delta_{7,5}$, which is shown on the right.}
\label{fig:Fdelta75}
\end{narrow}
\end{figure}

We prove this claim using structural induction on the Stern-Brocot tree $\W_2$.  The proof of the base case is shown in \Cref{fig:pps21}.  For the inductive step, we consider \Cref{fig:exp1}.  Let $(p,q)$ correspond to the top diagram in \Cref{fig:exp1}, and let $(p',q')$ and $(p'',q'')$ correspond to the diagrams on the lower left and lower right respectively.

\begin{figure}[htbp]
\centering
\includegraphics[width=\textwidth]{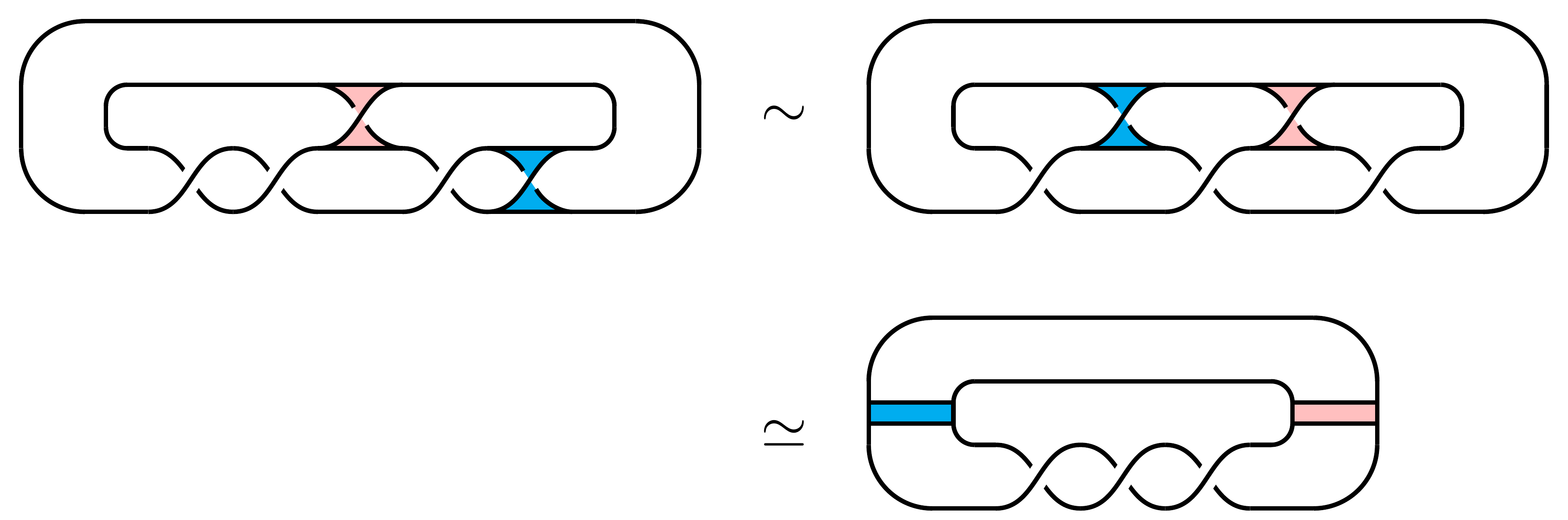}
\begin{narrow}{0.3in}{0.3in}
\caption
{{\bf The proof of \Cref{thm:pps} for $p=2$, $q=1$.} The second diagram is obtained from the first by sliding the blue band over the pink one, and the third is isotopic to the second.}
\label{fig:pps21}
\end{narrow}
\end{figure}

By induction, $F'_{p,q}$ represents the same surface as $F^\delta_{p,q}$.  In particular, the blue bands in the bottom right of the given diagram of $F'_{p,q}$ may be moved by a sequence of band slides to give the blue bands in the top left of the given diagram of $F^\delta_{p,q}$. 

In the case of $(p'',q'')$, we begin with a band slide as shown on the left side of \Cref{fig:ppsslides}.
In both cases, we then move any blue bands inherited from the $(p,q)$ diagram, using the fact that all three diagrams in \Cref{fig:exp1} become isotopic after applying the pink band move.  Finally we isotope the pink band as shown on the right side of \Cref{fig:ppsslides}.

\begin{figure}[htbp]
\centering
\includegraphics[width=\textwidth]{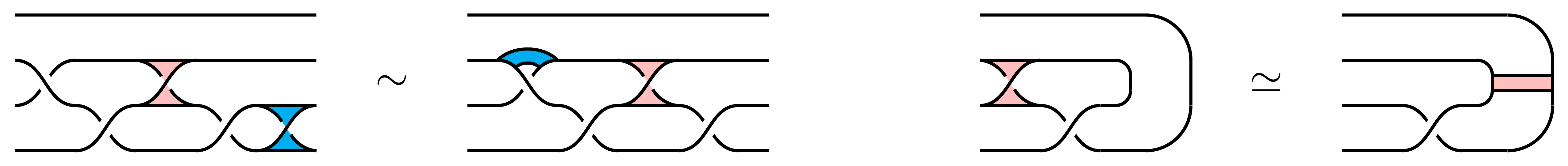}
\begin{narrow}{0.3in}{0.3in}
\caption
{{\bf A band slide and an isotopy.} }
\label{fig:ppsslides}
\end{narrow}
\end{figure}

We have now established that the knot with bands $F'_{p,q}$ represents the surface $F^\delta_{p,q}$ in the 4-ball.  Applying all the blue band moves first, we see the slice surface $\Delta_{p,q}$ as a sublevel surface, as required.
\end{proof}

\begin{proof}[Proof of \Cref{thm:Kh13}]
Let $p\ge 3$.  This proof is based on the diagram in \Cref{fig:Kh13}.  The reader may verify that performing the red band move yields the two-component unlink, while performing the blue band move converts the link to $K_{2,1}$.  We will see that this diagram may be interpreted as the middle frame of a movie exhibiting $\Delta_{p,1}$ as a sublevel set of $F_{-4}$, if $p$ is odd, or of the boundary sum $\Delta_{2,1}\natural P_-$ if $p$ is even.

\begin{figure}[htbp]
\centering
\includegraphics[width=10cm]{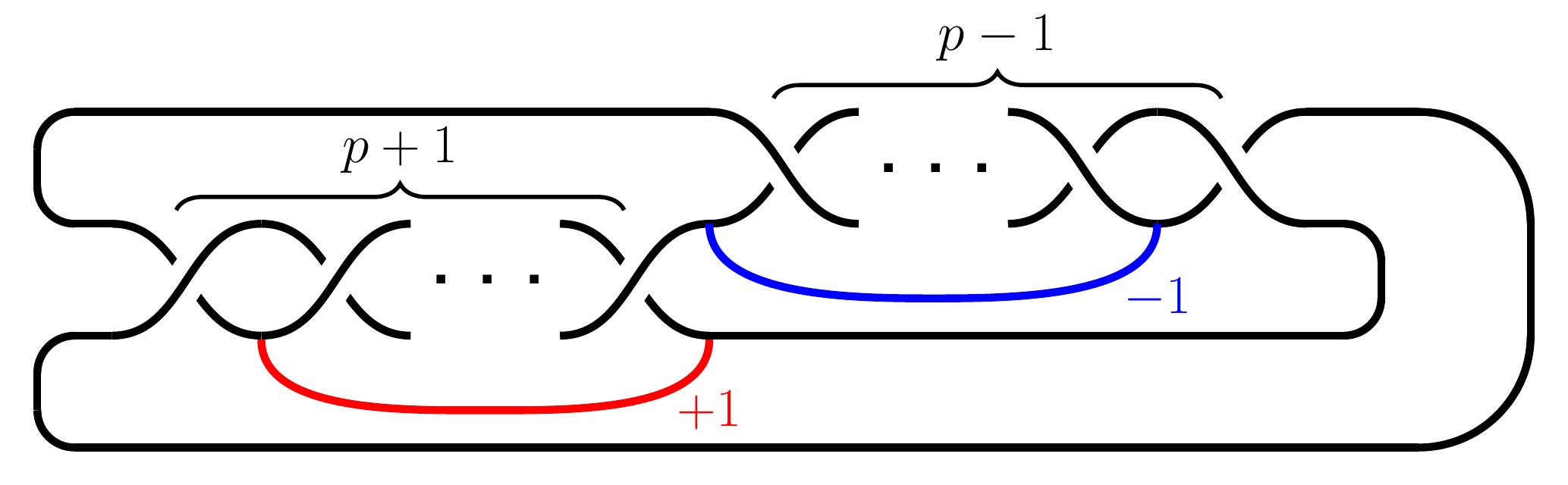}
\begin{narrow}{0.3in}{0.3in}
\caption
{{\bf The middle frame of a movie.} Here we use arcs to represent bands; the label beside each arc is the signed number of crossings in the band.  We may think of the $(+1)$-labelled red band move as decreasing the radial distance function and the blue one as increasing it.}
\label{fig:Kh13}
\end{narrow}
\end{figure}

\Cref{fig:Kh13isos} shows a diagram consisting of the link $K_{2,1}=S(4,1)$ together with two bands; after performing the indicated isotopies and the blue band move, this is seen to contain $\Delta_{p,1}$ as a sublevel surface.  It remains to simplify this diagram.  The first diagram in \Cref{fig:Kh13isos} shows two nested bands attached near the ends of a twist region.  There are $p$ crossings between the ends of one of them, and $p+2$ between the ends of the other.

\begin{figure}[htbp]
\centering
\includegraphics[width=\textwidth]{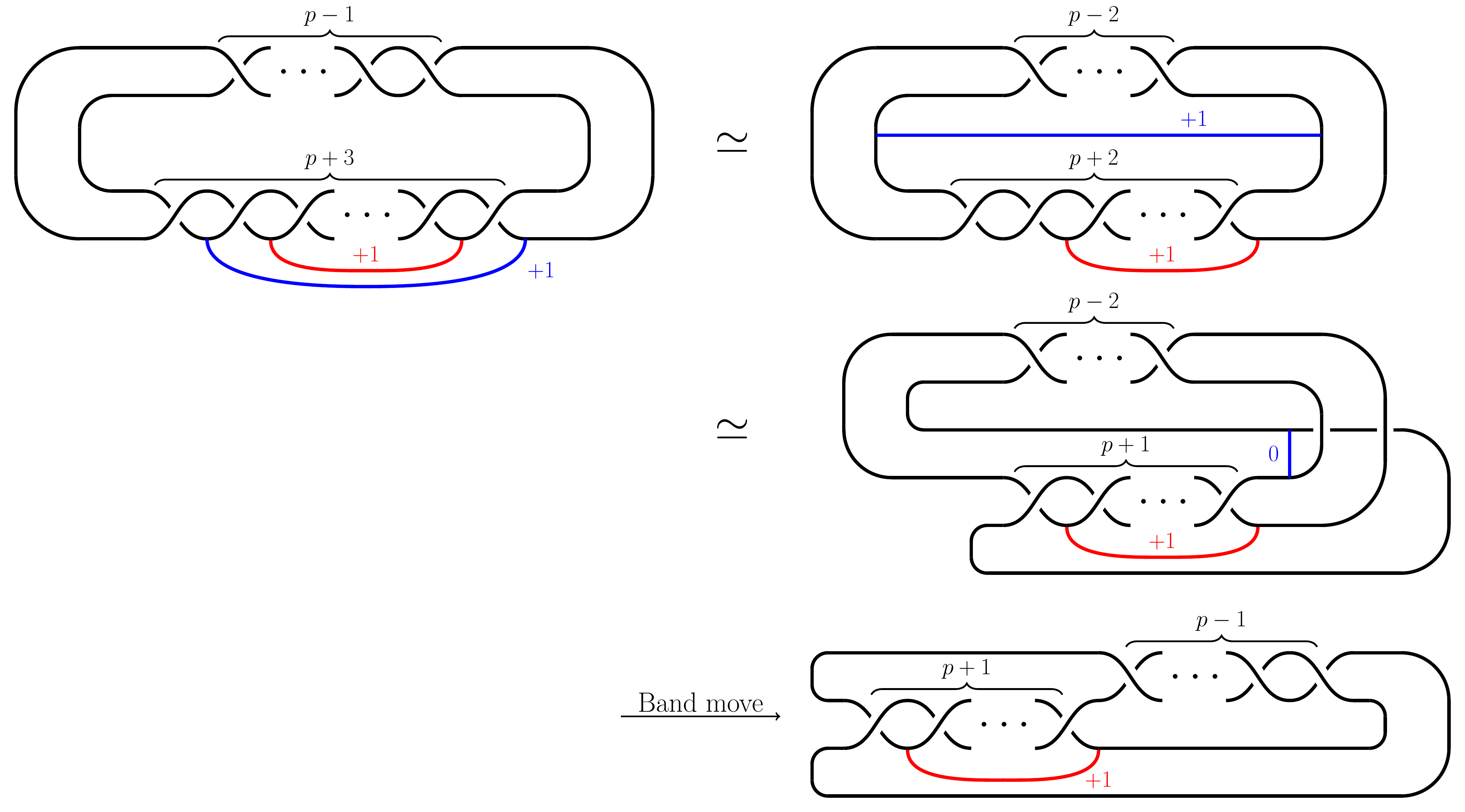}
\begin{narrow}{0.3in}{0.3in}
\caption
{{\bf A surface bounded by $K_{2,1}$.} We see that this contains $\Delta_{p,1}$ as a sublevel surface.}
\label{fig:Kh13isos}
\end{narrow}
\end{figure}

Using a band swim of the form \includegraphics[height=1cm]{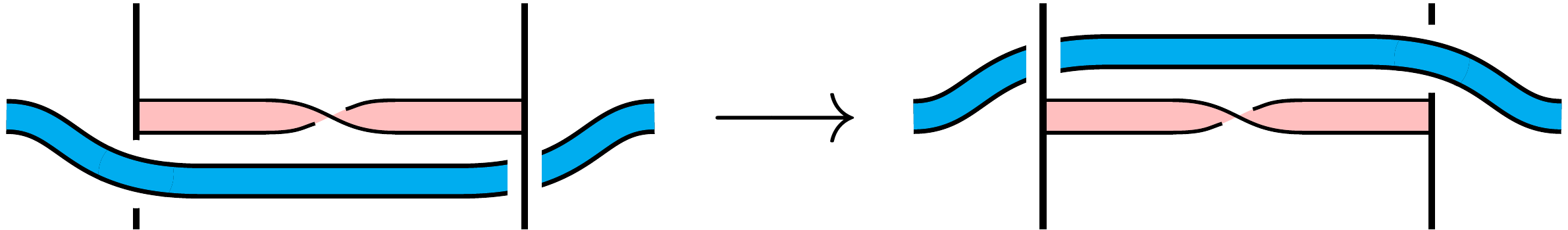}, we can move the blue band inside the red one, so that now one of the pair has $p$ crossings between its ends and the other has $p-2$ crossings.  Iterate until we finish up with one of the two diagrams shown in \Cref{fig:Khfinal};
thus we have realised $\Delta_{p,1}$ as a sublevel surface of $F_{-4}$ if $p$ is odd and of $\Delta_{2,1}\natural P_-$ is $p$ is even.

\begin{figure}[htbp]
\centering
\includegraphics[width=\textwidth]{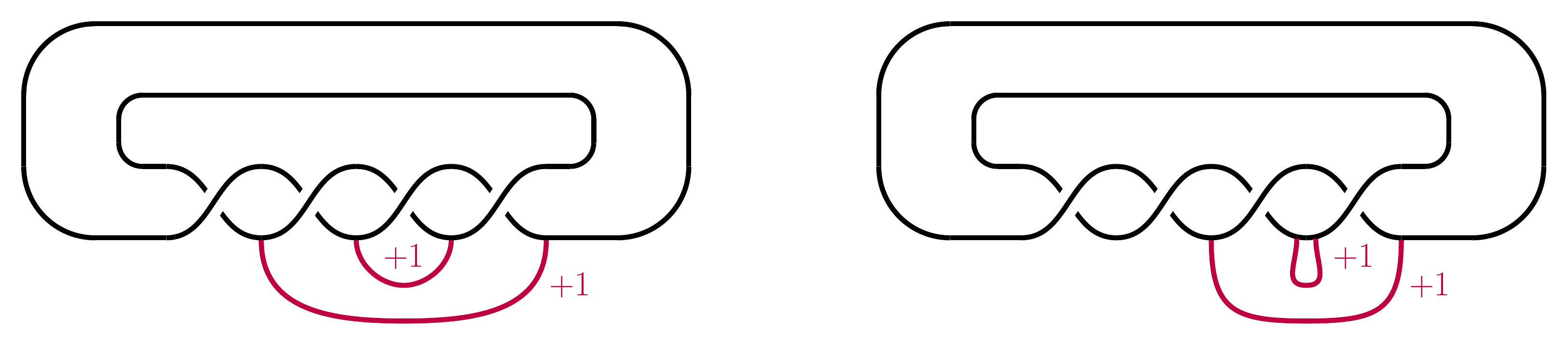}
\begin{narrow}{0.3in}{0.3in}
\caption
{{\bf Two surfaces bounded by $K_{2,1}$.} The surface on the left is $F_{-4}$, and the one on the right is the boundary sum of $\Delta_{2,1}$ and $P_-$ (cf. Figures \ref{fig:Delta21} and \ref{fig:FP}).}
\label{fig:Khfinal}
\end{narrow}
\end{figure}

\end{proof}

\begin{proof}[Proof of \Cref{thm:psquared}] This is very similar to the proof of \Cref{thm:Kh13} so we omit some details.  \Cref{fig:psquared} shows a diagram consisting of the link $K_{p^2,p-1}$ together with two bands; performing the $(+1)$-labelled red band move yields the two-component unlink, while performing the blue band move converts the link to $K_{p,1}$.  The knot with bands that we need is given by performing the blue band move and drawing in its inverse, i.e., the band which undoes the previous band move.  This gives two $(+1)$-labelled bands which are nested in a similar way to those encountered in the proof of \Cref{thm:Kh13}.  Performing $p$ band swims as in that proof results in the diagram of \Cref{fig:psquaredfinal} representing the boundary sum of $\Delta_{p,1}$ and $P_-$.
\end{proof}

\begin{figure}[htbp]
\centering
\includegraphics[width=\textwidth]{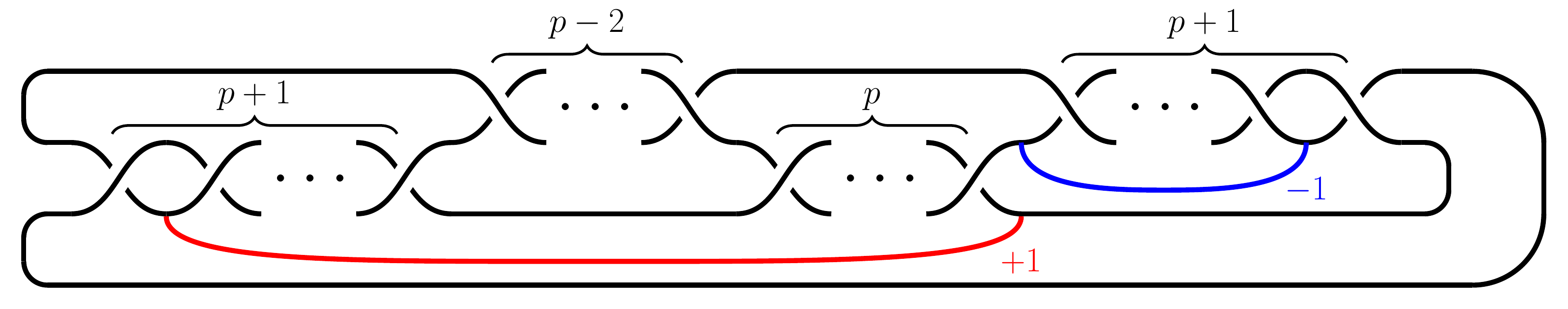}
\begin{narrow}{0.3in}{0.3in}
\caption
{{\bf The middle frame of another movie.}}
\label{fig:psquared}
\end{narrow}
\end{figure}

\begin{figure}[htbp]
\centering
\includegraphics[width=11cm]{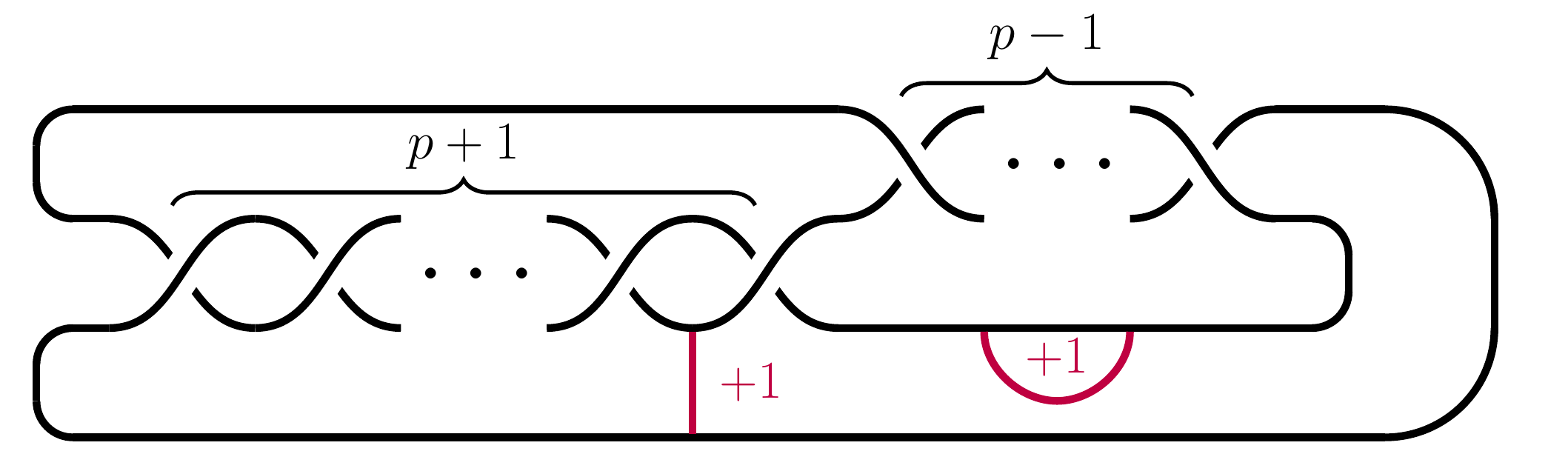}
\begin{narrow}{0.3in}{0.3in}
\caption
{{\bf The surface $\Delta_{p,1}\natural P_-$.}}
\label{fig:psquaredfinal}
\end{narrow}
\end{figure}

\begin{proof}[Proof of \Cref{thm:CP2bar}]
The usual recursive definition of the Fibonacci numbers easily implies
$$F(2n+2)=3F(2n)-F(2n-2),$$
which by an easy induction gives
$$\frac{F(2n+2)}{F(2n)}=[3^n].$$
Now using \Cref{fig:planardual} we have
$$\frac{F(2n+2)}{F(2n+2)-F(2n)}=[2,3^{n-1},2],$$
and then \Cref{lem:cf} yields
$$\frac{F(2n+2)^2}{F(2n+2)F(2n)-1}=[3^{n-1},5,3^{n-1},2].$$

Consider now the surface embedded in $B^4$ depicted in \Cref{fig:CP2bar}.  We see that the $(-1)$-labelled blue band move converts this into a diagram of $\Delta_{F(2n+2),F(2n)}$, which is thus a sublevel surface.  Also observe that the boundary link is an unknot: to see this perform a sequence of $2n$ isotopies of the form \includegraphics[width=4cm]{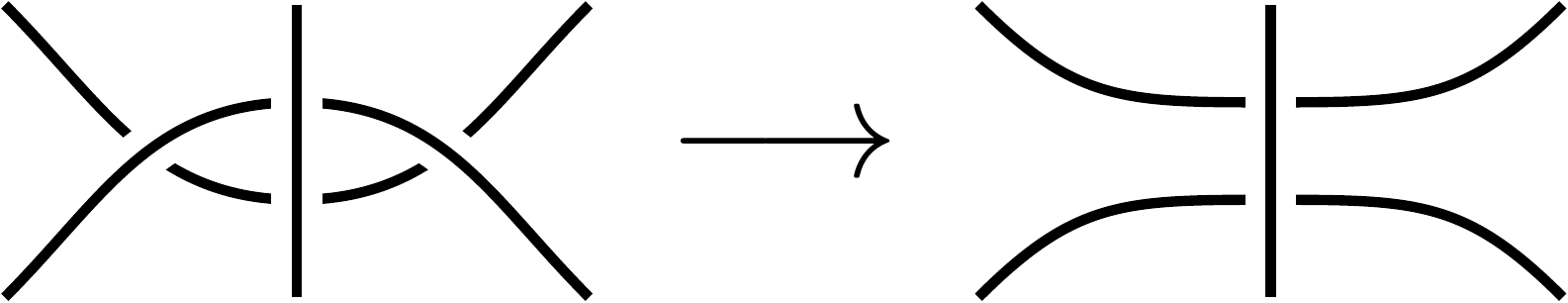}, followed by a Reidemeister II and a Reidemeister I move.  (The first of these isotopies involves the four crossings to the immediate left of the red band; in general these $2n$ isotopies can be located by noting that they always involve the single nonalternating crossing, which in the initial diagram is the one to the left of the blue band.) The proof is then completed by the sequence of handleslides shown in \Cref{fig:CP2barslides}.
\end{proof}

\begin{figure}[htbp]
\centering
\includegraphics[width=\textwidth]{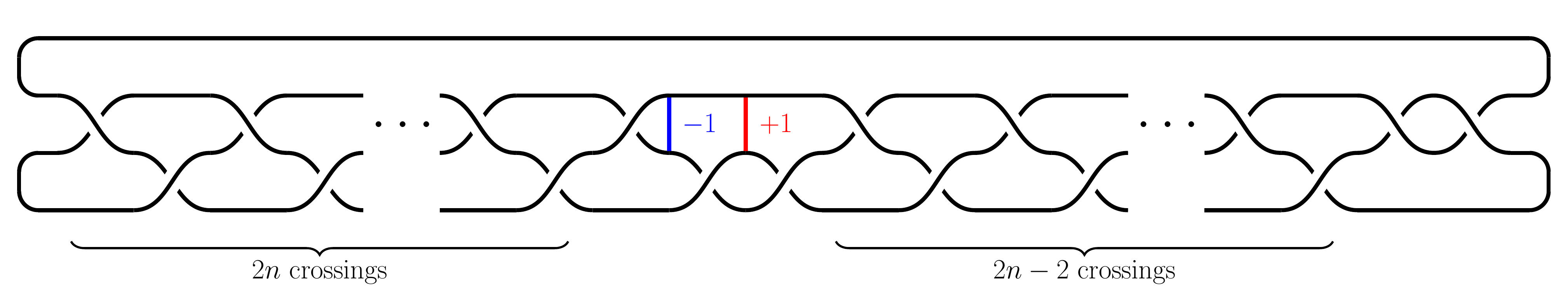}
\begin{narrow}{0.3in}{0.3in}
\caption
{{\bf A surface bounded by the unknot.}}
\label{fig:CP2bar}
\end{narrow}
\end{figure}

\begin{figure}[htbp]
\centering
\includegraphics[width=\textwidth]{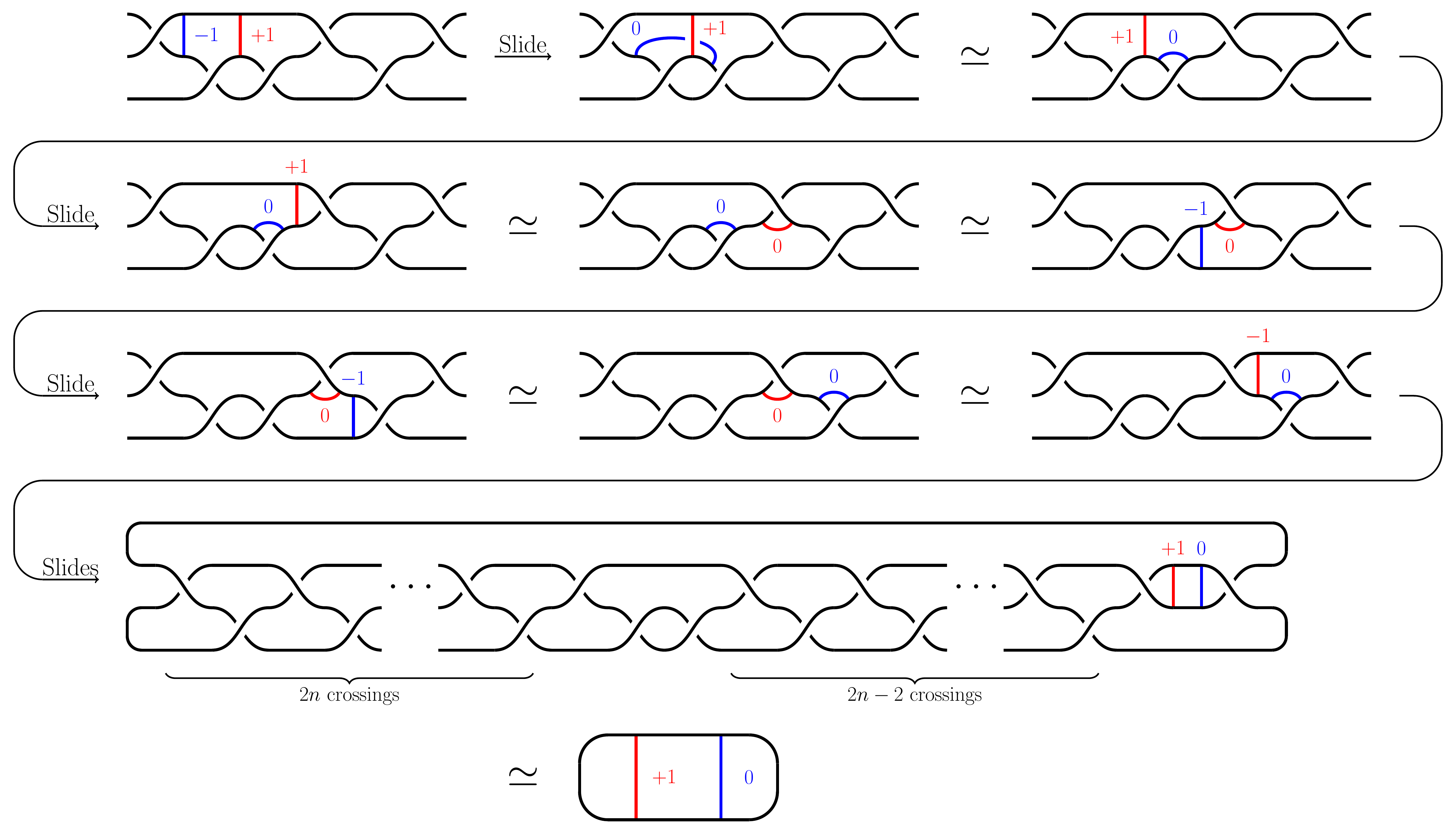}
\begin{narrow}{0.3in}{0.3in}
\caption
{{\bf Simplifying a surface bounded by the unknot.}}
\label{fig:CP2barslides}
\end{narrow}
\end{figure}

\begin{remark}
Applying almost the same proof for odd Fibonacci numbers realises $\Delta_{F(2n+1),F(2n-1)}$ as a sublevel surface of $P_+$, and hence a smooth embedding of $B_{F(2n+1),F(2n-1)}$ in $\CP^2$.  It is well-known that in fact an embedding of $B_{F(2n+1),F(2n-1)}$ in $\CP^2$ arises from singularity theory since $(1,F(2n-1),F(2n+1))$ is a Markov triple, cf. \cite{aigner,es,hp}.
\end{remark}


\section{Simple embeddings}
\label{sec:simple}

In \cite{kho}, Khodorovskiy called an embedding $B_{p,q}\hookrightarrow Z$ {\em simple} if the corresponding rational blow-up $X=Z\setminus B\,\cup_{Y_{p,q}} C_{p,q}$ may be obtained from $Z$ by a sequence of ordinary blow-ups, or in other words if $X=Z\#k\,\overline{\CP^2}$.  She pointed out in that paper that the embeddings in \Cref{thm:Kh12}, and those in \Cref{thm:Kh13} with odd $p$, are simple, and Park-Park-Shin showed that the embeddings  described in \Cref{thm:pps} are also simple.  One could extend the notion of simple to embeddings of the form $B_{p,q}\hookrightarrow B_{p',q'}\# \overline{\CP^2}$ by saying that such an embedding is simple if the resulting rational blow-up of $B_{p,q}$ has the same effect as rationally blowing up $B_{p',q'}$, together with a sequence of ordinary blow-ups.  With this terminology the proof of \cite[Corollary 5.1]{kho} applies to show that the embeddings in \Cref{thm:Kh13} are all simple.

Given a sublevel surface $\Delta_{p,q}$ of a properly embedded surface $F$ in the 4-ball, the {\em equivariant rational blow-up} of $\Delta_{p,q}$ is the surface obtained by replacing $\Delta_{p,q}$ by the pushed-in black surface of its boundary link.  Examples are shown in Figures \ref{fig:PPSsimple}, \ref{fig:Kh13simple}, \ref{fig:psquaredsimple}, and \ref{fig:CP2barsimple}.
We say a sublevel surface $\Delta_{p,q}$ of a properly embedded surface $F$ in the 4-ball is simple if the equivariant rational blow-up yields $F\natural kP_-$ for some $k\in\nn$, or if $F=\Delta_{p',q'}$ and the equivariant rational blow-up yields $F_{p',q'}\natural kP_-$, where $F_{p',q'}$ is the pushed-in black surface of $K_{p',q'}$.

\begin{proposition}  
\label{prop:simple}
All of the sublevel surfaces $\Delta_{p,q}\hookrightarrow F$ exhibited in the previous section are simple.
\end{proposition}

\begin{proof}
The proofs are fairly straightforward and it suffices to illustrate with examples.  For the case of Theorems \ref{thm:Kh12} and \ref{thm:pps}, see \Cref{fig:PPSsimple}.  For \Cref{thm:Kh13}, see \Cref{fig:Kh13simple}.  For \Cref{thm:psquared}, see \Cref{fig:psquaredsimple}. Finally for \Cref{thm:CP2bar}, see \Cref{fig:CP2barsimple}.

\end{proof}

\begin{remark}  It is very interesting to note that Finashin \cite{fin} has given examples of the use of equivariant rational blow-down in the construction of  exotic smooth manifolds homeomorphic to $\CP^2\#5\overline{\CP^2}$.
\end{remark}

\begin{figure}[htbp]
\centering
\includegraphics[width=\textwidth]{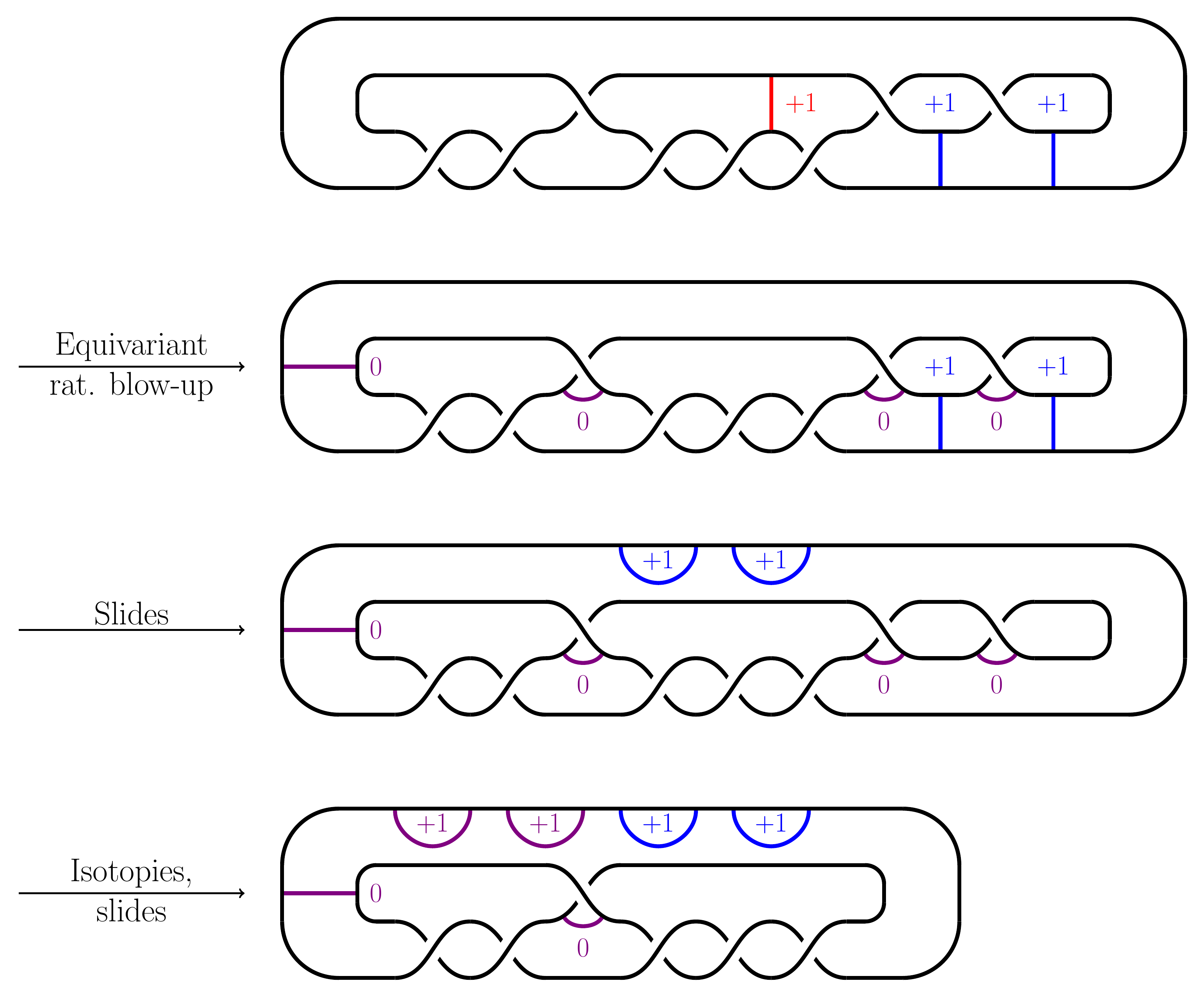}
\begin{narrow}{0.3in}{0.3in}
\caption
{{\bf Simplicity of the sublevel surface $\Delta_{8,3}\hookrightarrow F^\delta_{8,3}$.}  The equivariant rational blow-up is achieved by replacing the band for $\Delta_{8,3}$ with those for the black surface of $K_{8,3}$.  To go from the second diagram to the third, slide the blue bands to the left, and then note that a $P_-$ boundary summand can be moved anywhere in the diagram; to move it past another band, slide the other band over the trivial $(+1)$-labelled band twice.}
\label{fig:PPSsimple}
\end{narrow}
\end{figure}

\begin{figure}[htbp]
\centering
\includegraphics[width=\textwidth]{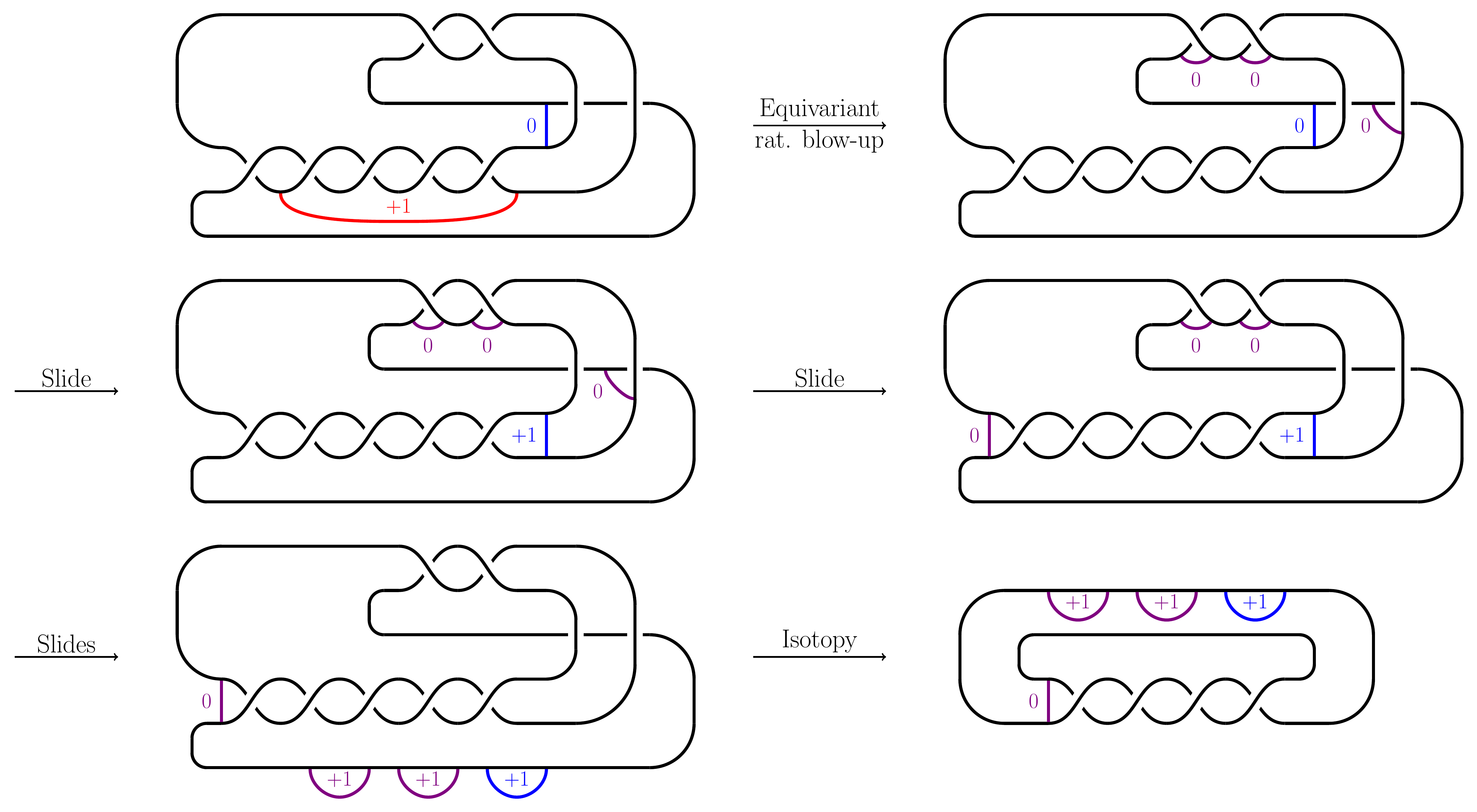}
\begin{narrow}{0.3in}{0.3in}
\caption
{{\bf Simplicity of the sublevel surface $\Delta_{4,1}\hookrightarrow \Delta_{2,1}\natural P_-$.}  }
\label{fig:Kh13simple}
\end{narrow}
\end{figure}

\begin{figure}[htbp]
\centering
\includegraphics[width=\textwidth]{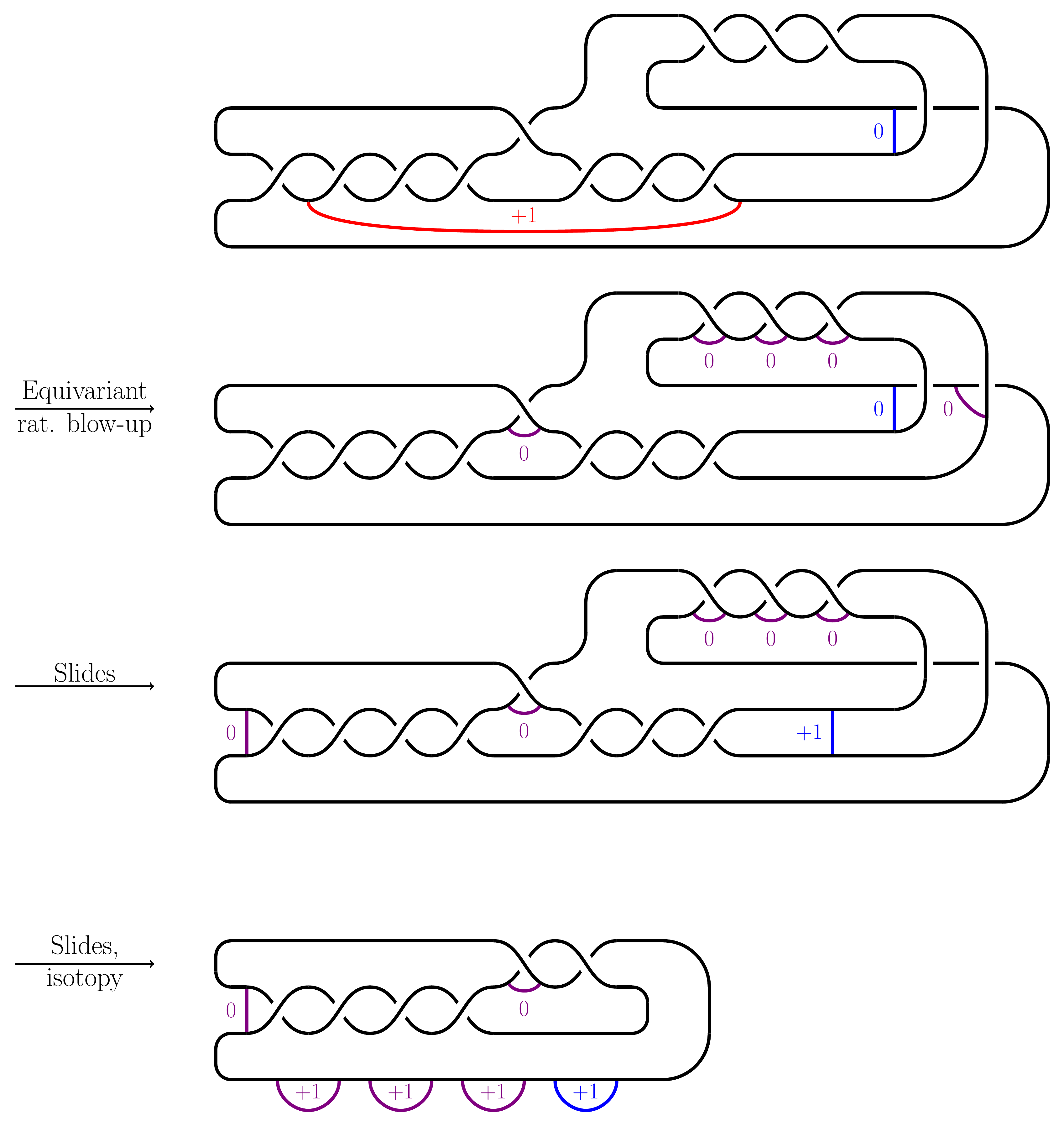}
\begin{narrow}{0.3in}{0.3in}
\caption
{{\bf Simplicity of the sublevel surface $\Delta_{9,1}\hookrightarrow \Delta_{3,1}\natural P_-$.}  }
\label{fig:psquaredsimple}
\end{narrow}
\end{figure}

\begin{figure}[htbp]
\centering
\includegraphics[width=\textwidth]{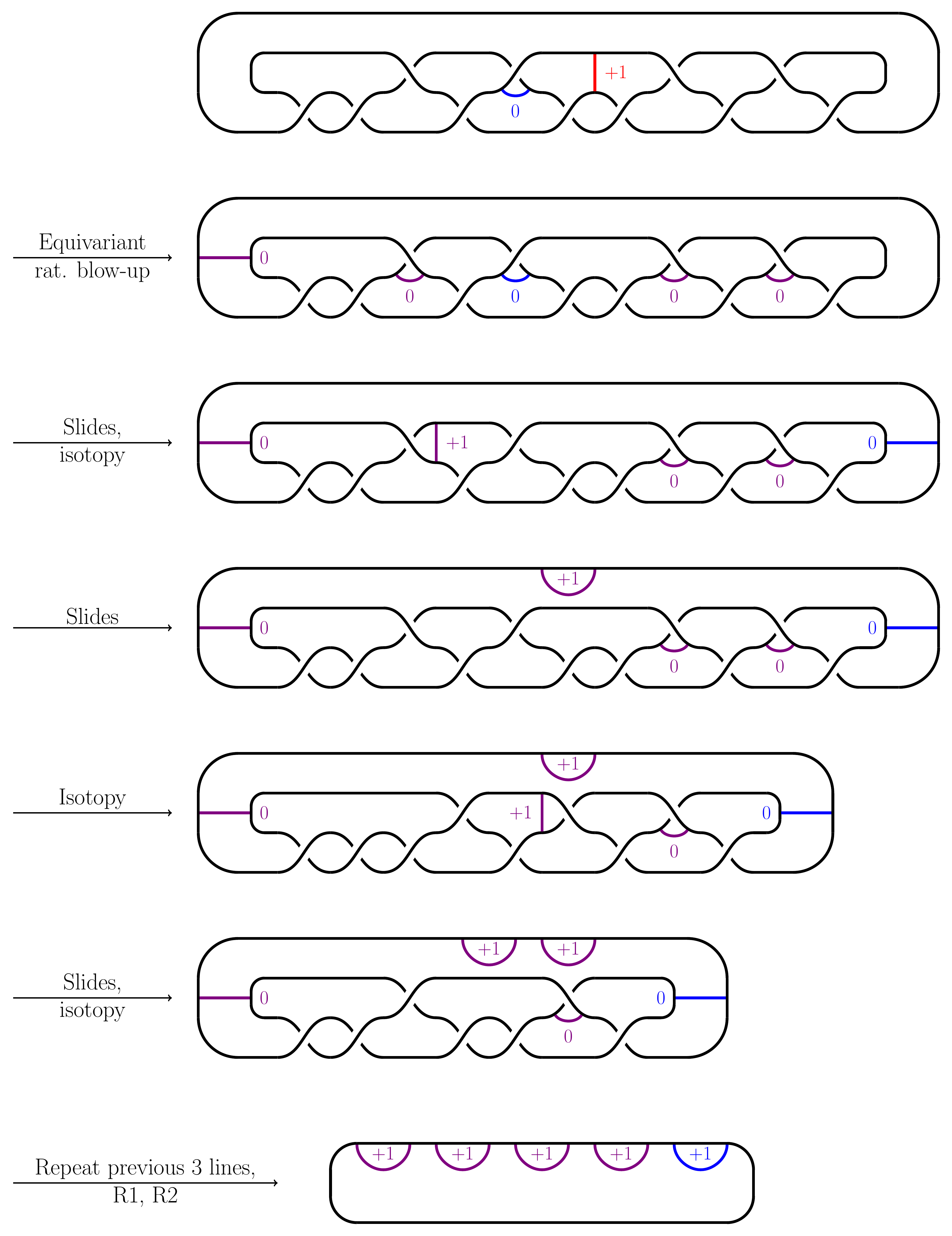}
\begin{narrow}{0.3in}{0.3in}
\caption
{{\bf Simplicity of the sublevel surface $\Delta_{8,3}\hookrightarrow P_-$.}  The sequence of slides and isotopies from the 3rd diagram to the 6th give an inductive step to go from $n$ to $n-1$.}
\label{fig:CP2barsimple}
\end{narrow}
\end{figure}


\clearpage

\bibliographystyle{amsplain}
\bibliography{balls}

\end{document}